\documentclass[11pt, twoside, letterpaper]{amsart}
\usepackage{amsmath, hyperref, color}

\newtheorem{thm}{Theorem}[section]
\newtheorem{cor}[thm]{Corollary}
\newtheorem{lem}[thm]{Lemma}
\newtheorem{prop}[thm]{Proposition}

\theoremstyle{definition}
\newtheorem{defn}[thm]{Definition}

\newtheorem{ass}[thm]{Assumption}

\theoremstyle{remark}
\newtheorem{rem}[thm]{Remark}

\numberwithin{equation}{section}

\newcommand{\set}[1]{\left\{#1\right\}}
\newcommand{\Real}{\mathbb R}
\newcommand{\Natural}{\mathbb N}
\newcommand{\R}{\mathbb{R}}
\newcommand{\N}{\mathbb{N}}

\newcommand{\such}{\ | \ }
\newcommand{\nin}{n \in \Natural}
\newcommand{\prob}{\mathbb{P}}

\newcommand{\Exp}{\mathcal E}
\newcommand{\cE}{\mathcal{B}_E}
\newcommand{\FF}{\bF}

\newcommand{\PP}{\mathbb{P}}
\newcommand{\QQ}{\mathbb{Q}}
\newcommand{\EE}{\mathbb{E}}

\newcommand{\qprob}{\mathbb{Q}}

\newcommand{\expec}{\mathbb{E}}

\newcommand{\probtriple}{(\Omega, \mathcal{F}, \prob)}

\newcommand{\basispwf}{(\Omega, \, \F, \, \bF, \, \prob)}
\newcommand{\basis}{(\Omega, \, \bF)}
\newcommand{\basisg}{(\Omega, \, \bG)}

\newcommand{\basisp}{(\Omega, \, \bF, \, \prob)}

\newcommand{\basisgp}{(\Omega, \, \bG, \, \prob)}

\newcommand{\basisq}{(\Omega, \, \bF, \, \qprob)}

\newcommand{\F}{\mathcal{F}}
\newcommand{\G}{\mathcal{G}}

\newcommand{\cadlag}{c\`adl\`ag}

\newcommand{\ud}{\mathrm d}

\newcommand{\cB}{\mathcal{B}}
\newcommand{\cO}{\mathcal{O}}
\newcommand{\cP}{\mathcal{P}}

\newcommand{\limn}{\lim_{n \to \infty}}

\newcommand{\zi}{{\Real_+}}

\newcommand{\e}{\mathrm{e}}

\newcommand{\cS}{{}^\mathsf{c} \kern-0.21em S}
\newcommand{\cH}{{}^\mathsf{c} \kern-0.23em H}
\newcommand{\cMM}{{}^\mathsf{c} \kern-0.19em [M, M]}

\newcommand{\pare}[1]{\left(#1\right)}
\newcommand{\bra}[1]{\left[#1\right]}
\newcommand{\dbra}[1]{[\kern-0.15em[ #1 ]\kern-0.15em]}
\newcommand{\dbraco}[1]{[\kern-0.15em[ #1 [\kern-0.15em[}
\newcommand{\dbraoc}[1]{]\kern-0.15em] #1 ]\kern-0.15em]}
\newcommand{\dbraoo}[1]{]\kern-0.15em] #1 [\kern-0.15em[}

\newcommand{\Xhat}{\widehat{X}}

\newcommand{\oF}{\overline{\F}}

\newcommand{\X}{\mathcal{X}}

\newcommand{\bF}{\mathbf{F}}

\newcommand{\Lb}{\mathbb{L}}

\newcommand{\dfn}{\, := \,}

\newcommand{\lz}{\Lb^0}
\newcommand{\lzp}{{\lz_+}}

\newcommand{\bG}{\mathbf{G}}
\newcommand{\GG}{\mathbf{G}}
\newcommand{\indic}{\mathbb{I}}
\newcommand{\ind}{\mathbb{I}}

\newcommand{\iii}{{i \in \set{1, \ldots, d}}}

\newcommand{\tir}{t \in \Real_+}
\newcommand{\Y}{\mathcal{Y}}

\newcommand{\Yfq}{\Y(\bF, S, \qprob)}

\newcommand{\Yfp}{\Y(\bF, S, \prob)}
\newcommand{\Ygp}{\Y(\bG, S^\tau, \prob)}

\newcommand{\naone}{NA$_1$}
\newcommand{\naonef}{NA$_1 (\bF, S)$}
\newcommand{\naoneg}{NA$_1 (\bG, S^\tau)$}

\newcommand{\be}{\begin{equation}}
\newcommand{\ee}{\end{equation}}
\newcommand{\ba}{\begin{aligned}}
\newcommand{\ea}{\end{aligned}}

\newcommand{\lsi}{\left[\negthinspace\left[\right.\right.\!}

\newcommand{\pj}{\gamma}
\newcommand{\oFF}{\overline{\FF}}
\newcommand{\Stilde}{\widetilde{S}}

\begin{document}

\title[Arbitrage of the first kind and filtration enlargements in semimartingale models]{Arbitrage of the first kind and filtration enlargements in semimartingale financial models}%

\author{Beatrice Acciaio}%
\address{Beatrice Acciaio, Statistics Department, London School of Economics and Political Science}%
\email{b.acciaio@lse.ac.uk}%

\author{Claudio Fontana}%
\address{Claudio Fontana, Laboratoire de Probabilit\'es et Mod\`eles Al\'eatoires, Paris Diderot University}%
\email{fontana@math.univ-paris-diderot.fr}%
\thanks{The research of the second author was supported by a Marie Curie Intra European Fellowship within the 7th European Community Framework Programme under grant agreement PIEF-GA-2012-332345.}%

\author{Constantinos Kardaras}%
\address{Constantinos Kardaras, Statistics Department, London School of Economics and Political Science}%
\email{k.kardaras@lse.ac.uk}%


\date{\today}%


\begin{abstract}
In a general semimartingale financial model, we study the stability of the \emph{No Arbitrage of the First Kind} (\naone) (or, equivalently, \emph{No Unbounded Profit with Bounded Risk}) condition under initial and under progressive filtration enlargements. In both cases, we provide a simple and general condition which is sufficient to ensure this stability for \emph{any fixed} semimartingale model. Furthermore, we give a characterisation of the \naone\ stability for \emph{all} semimartingale models.
\end{abstract}

\maketitle

\section*{Introduction}

In financial mathematics, market models with different sets of information have been widely studied, especially in relation to insider trading and credit risk modeling (see e.g. \cite{JYC09} and the references therein). Typically, one starts by postulating a model with respect to a given information set and then enlarges that set with some additional information not originally present in the market. From a mathematical point of view, this corresponds to considering an \emph{enlargement} of the original filtration on a given filtered probability space. Since the model aims at representing a financial market, a fundamental question is whether the additional information allows for arbitrage profits.

The present paper aims at answering the above question in the context of models driven by general semimartingales, both in the case where the additional information is added in a progressive way through time, and in the case where the additional information is fully added at the initial time. Referring to the terminology of the theory of enlargement of filtrations (see \cite{MR604176} for a complete account of the theory and \cite[\S~5.9]{JYC09} and \cite[Ch.~VI]{MR1037262} for a presentation of the main results), this corresponds to considering a filtration obtained as a \emph{progressive} or as an \emph{initial} enlargement, respectively, of the original filtration.

Our analysis focuses on the \emph{No Arbitrage of the First Kind} (\naone) condition (see \cite{Kar_09_fin_add_ftap}), which is equivalent to the \emph{No Unbounded Profit with Bounded Risk} (NUPBR) condition (see \cite[Proposition 1]{Kar_09_fin_add_ftap}). Mathematically, condition \naone \ is equivalent to existence of strictly positive local martingale deflators, and can be shown to be the minimal condition ensuring the well-posedness of expected utility maximisation problems (see \cite[Proposition 4.19]{MR2335830}). In the case of a progressive enlargement with respect to a random time $\tau$, we study the stability of \naone\ on the random time horizon $[0,\tau]$, showing that the existence of arbitrages of the first kind in the enlarged filtration is crucially linked to the possibility of the asset-price process exhibiting a jump at the same time when a particular nonnegative local martingale in the original filtration jumps to zero.
In turn, we show that the possibility of the latter event is intimately related to how local martingales from the original filtration behave in the enlarged filtration, up to a suitable normalisation. 
In the case of an initial enlargement of the original filtration, and under the classical density hypothesis of \cite{Jac85}, we establish an analogous set of results, showing that the validity of \naone\ in the enlarged filtration is linked to the possibility of the asset-price process jumping at the same time when a family of nonnegative martingales in the original filtration jumps to zero. In turn, as in the case of progressive enlargements, the latter possibility also fully characterises how local martingales from the original filtration behave in the enlarged filtration, up to a suitable normalisation.
 
In both cases of progressive and of initial enlargement, these results allow us to provide an easy sufficient condition ensuring the \naone\ stability for \emph{a fixed} semimartingale model, as well as to explicitly characterise the stability of \naone\ for \emph{all} semimartingale models. Although absent in the statements of our main results, an inspection of their proofs reveals a hands-on approach to the problem: using local martingale deflators in the original filtration, we explicitly construct local martingale deflators in the enlarged filtration in order to show validity of condition \naone.
In the process, we obtain some interesting new results on progressive as well as initial filtration enlargement, showing how the super/local martingale property of a process can be transferred from the original filtration to the enlarged one by suitably deflating the process.

For progressive filtration enlargement with respect to an honest time $\tau$ (see \cite[Ch.~VI]{MR1037262}), examples of arbitrage profits are provided in \cite{Imk02}, \cite{Zw07} and \cite{FJS}. In the context of continuous semimartingale models, as shown in \cite[Theorem 4.1]{FJS} (see also \cite[Lemma 6.7]{Kre13}), condition  \naone\ is always valid in the enlarged filtration on the random time horizon $[0,\tau]$. In the case of general semimartingale models, this is no longer true, see the example in \S~\ref{subsubsec: inacc}. 
In that context, the recent paper \cite{Ch_et_al_13} addresses the issue of \naone\ stability in progressively enlarged filtrations and represents one of the sources of inspiration for the present work. 
In particular, the key role of conditions equivalent to those given in Theorem~\ref{thm: NA1_G} and Remark~\ref{rem: NA1_G} has been first pointed out and proved in \cite{Ch_et_al_13} (see Remark~\ref{rem:comp_bis}) and the characterisation we obtain in Theorem~\ref{thm:gen_progr} turns out to be equivalent to the one already established in \cite{Ch_et_al_13} (see Remark~\ref{rem:comp}).
However, in comparison with the latter paper, we follow here a totally different approach and provide original and rather simple proofs to those results, avoiding the use of the compensated stochastic integral (see e.g. \cite[Definition~9.7]{MR1219534}) and, somewhat surprisingly, not relying on the classical Jeulin-Yor decomposition formula (see \cite[Proposition 4.16]{MR604176}). In contrast, we exploit the properties of an optional decomposition of the Az\'ema supermartingale associated to $\tau$ recently established in \cite{Kar_14}.
We also want to mention that, in the case of the classical \emph{No Free Lunch with Vanishing Risk} (NFLVR) condition (see \cite{MR1304434,MR1671792}), a study of its stability and of the relation with the preservation of the martingale property in progressively enlarged filtrations has been carried out in \cite{CJN12}.

In the initial filtration enlargement case, the possibility of realising arbitrage profits in the enlarged filtration has been studied in \cite{GP98}, \cite{MR1831271} and \cite{IPW01}, among others. Concerning the classical NFLVR condition, it is well-known that it is stable under an initial enlargement with respect to a random variable $J$ if the conditional law of $J$ for all times is equivalent to the unconditional one (see e.g. \cite{GP98}). However, to the best of our knowledge, the issue of \naone\ stability with respect to an initial enlargement has never been studied so far.
Interestingly, we show that both the progressive and the initial case can be treated by relying on the same methodological approach.

The paper is organised as follows. Section~\ref{sect.main} contains the framework and statements of our main results. 
In Section~\ref{sect.progr} we consider progressive enlargement of filtrations. We study the crucial stopping times that will be then used to pinpoint local martingales and to prove stability of the \naone\ condition in the enlarged filtrations.
In Section~\ref{sect.init} we perform the same analysis and obtain analogous results, \emph{mutatis mutandis}, in the case of initially enlarged filtrations.

\section{Main Results}\label{sect.main}

\subsection{Probabilistic set-up} \label{subsec:prob_set_up}
In all that follows, we work on a filtered probability space $\basispwf$, where $\bF = (\F_t)_{\tir}$ is a filtration satisfying the usual hypotheses of right-continuity and saturation by $\prob$-null sets. In general, $\F_\infty \subseteq \F$ holds, with the last set-inclusion being potentially strict.

We shall be using standard notation from the general theory of stochastic processes. For any unexplained notation and results, the reader can consult \cite{MR1219534} or \cite{MR1943877}.

\subsection{The market model} \label{subsec:market_model}
Fix $d \in \Natural = \set{1, 2, \ldots}$, and let $S \equiv (S^i)_\iii$ be a collection of nonnegative semimartingales on $\basisp$ 
\footnote{We want to mention that the nonnegativity assumption is not crucial for the following results to hold, provided that the notion of local martingale is suitably replaced by the notion of sigma-martingale (see \cite{MR1671792} and \cite{Tak_13}).}. 
Each $S^i$, $\iii$, models the price process of an asset, discounted by a baseline security in the market. Starting with initial capital $x \in [0, \infty)$ and following a $d$-dimensional, $\bF$-predictable and $S$-integrable strategy $H$, an investor's discounted wealth process is given by $X^{x, H} \dfn x + \int_0^\cdot \pare{H_t, \ud S_t}$. It should be noted that we are using vector stochastic integration throughout. Define $\X(\bF, S)$ to be the class of all \emph{nonnegative} processes $X^{x, H}$ in the previous notation. (In the definition of the class $\X(\bF, S)$, the initial capital $x \in [0, \infty)$ and $d$-dimensional, $\bF$-predictable and $S$-integrable strategies $H$ are arbitrary, as long as $X^{x, H} \geq 0$.)

\begin{defn}	\label{def:A1}
For $T \in (0 ,\infty)$, an \textsl{arbitrage of the first kind} with information $\bF$ and assets $S$ on $[0, T]$ is $\chi_T \in \lzp (\F_T)$ with $\prob \bra{\chi_T > 0} > 0$ and with the property that for all $x \in (0, \infty)$ there exists $X \in \X(\bF, S)$ with $X_0 = x$ (where the wealth process $X$ may depend on $x$) such that $\prob \bra{X_T \geq \chi_T} = 1$. If no arbitrage of the first kind with information $\bF$ and assets $S$ exists on any interval $[0, T]$ for $T \in (0 ,\infty)$, we say that condition \naonef \ holds.
\end{defn}

Whenever $\qprob \sim \prob$, we use $\Yfq$ to denote the class of all strictly positive $\bF$-adapted \cadlag \ processes $Y$ with $Y_0 = 1$, such that $Y$ and $Y S$ are local martingales on $\basisq$. The elements in $\Yfq$ are called \emph{strictly positive local martingale deflators} (for $S$ on $\basisq$). When strict positivity is replaced by nonnegativity, we simply talk of \emph{local martingale deflators}. If $Y^{\qprob}$ denotes the density process of $\qprob$ with respect to $\prob$, note that $\Yfq = \big \{ Y {(Y^\qprob_0 / Y^\qprob)}  \such Y \in \Yfp \big\}$ holds.
It comes as a consequence of \cite[Theorem 2.6]{Tak_13} that condition \naonef \ is equivalent to $\Yfq \neq \emptyset$ (where, of course, $\qprob \sim \prob$ is arbitrary). For our purposes (see Remark~\ref{rem:explanation} below), we need a more precise statement.

\begin{thm} \label{thm: num_loc_mart}
Condition \emph{\naonef} holds if and only if there exist $\qprob \sim \prob$ and strictly positive $\Xhat \in \X(\bF, S)$ such that $(1 / \Xhat) \in \Yfq$.
\end{thm}

Note that, even though the statement of Theorem~\ref{thm: num_loc_mart} is sharper than~\cite[Theorem 2.6]{Tak_13}, it actually follows from the proof of the latter. Indeed,~\cite{Tak_13} prove that \naonef~implies the existence of $\qprob \sim \prob$ and strictly positive $\Xhat \in \X(\bF, S)$ such that $Y^{\qprob}/(\Xhat Y^{\qprob}_0)\in\Yfp$, with $Y^{\qprob}$ denoting the density process of $\qprob$ with respect to $\prob$.

The main purpose of the paper is the study of stability of the \naone\ condition when enlarging the filtration $\bF$ in a progressive or initial way. Naturally, the first issue to be settled is the preservation of the semimartingale property of processes, which is typically referred to in the literature as the $\mathcal{H}'$-hypothesis. In the case of progressive filtration enlargement by a random time $\tau$, it comes as a consequence of the Jeulin-Yor theorem that this always holds up to time $\tau$ (and that for honest times it holds on all $[0,\infty)$); see \cite{MR519998}. 
For the case of initial filtration expansion, one well-known situation where the preservation of the semimartingale property holds is when Jacod's density hypothesis is satisfied; see \cite{Jac85}. We want to remark that these facts will also come as consequences of our analysis in Section \ref{sect.progr} for the progressive enlargement case (see Corollary \ref{cor:H_prog}) and Section \ref{sect.init} for the initial enlargement case (see Remark \ref{rem:H_init}).

\begin{rem}	\label{rem:explanation}
Theorem~\ref{thm: num_loc_mart} will play a key role in the proof of our main theorems. In fact, it shows that \naone$(\bF,S)$ is equivalent to the existence of  $Y\in\Yfq$ such that $\set{\Delta S\neq0}=\set{\Delta Y\neq0}$, for some $\qprob \sim \prob$. As shown below (see Sections~\ref{subsec:proof_main_prog} and \ref{subsec:proofs_initial}), this property turns out to be crucial in order to construct local martingale deflators in enlarged filtrations starting from local martingale deflators from the original filtration (compare also with  Remarks \ref{rem:NA1_progr_takaoka} and \ref{rem:NA1_init_takaoka}).
\end{rem}

\subsection{Main results under progressive filtration enlargement}\label{subsec:main-progr}

We first study the stability of the \naone\ condition under a progressive enlargement of the filtration $\bF$ with respect to  an $\F$-measurable random time $\tau : \Omega \mapsto [0, \infty]$ such that $\prob \bra{\tau = \infty} = 0$. (We refer the reader to \cite[Chapter VI]{MR1037262} for a textbook account of the theory of enlargement of filtrations.) The progressively enlarged filtration $\bG = (\G_t)_{\tir}$ is defined via
\begin{equation}	\label{Gfiltr}
\G_t = \set{B \in \F \such B \cap \set{\tau > t} = B_t \cap \set{\tau > t} \text{ for some } B_t \in \F_t}, \quad \forall \tir.
\end{equation}
In particular, $\bG$ is a right-continuous filtration that contains $\bF$ and makes $\tau$ a stopping time, but note that $\bG$ is \emph{not} the smallest right-continuous filtration that contains $\bF$ and makes $\tau$ a stopping time, compare e.g. the discussion in \cite{GuoZeng}.

It comes as a consequence of the Jeulin-Yor theorem that $S^\tau \dfn \pare{S_{\tau \wedge t}}_{\tir}$ is a semimartingale on $\basisgp$ (see, for example, \cite{MR519998}; actually, we shall provide an alternative simple proof of this fundamental fact in Corollary~\ref{cor:H_prog}).
Then, the class $\X(\bG, S^\tau)$ can be defined exactly in the same way as the corresponding class $\X(\bF, S)$ of \S~\ref{subsec:market_model}. The notation \naoneg\ used in the sequel refers to absence of arbitrage of the first kind with information $\bG$ and assets $S^\tau$.

A key role in the study of progressive enlargement of filtrations is played by the Az\'ema supermartingale associated with $\tau$ (given by the optional projection of $\indic_{\dbraco{0,\tau}}$ on $\basisp$, see \cite{MR604176} and references therein), that we denote by $Z$. This means that $\prob \bra{\tau > \sigma \such \F_{\sigma}} = Z_{\sigma}$ for all finite stopping times $\sigma$ on $\basis$, and note that $Z_\infty := \lim_{t\rightarrow\infty}Z_t = 0$ holds in view of $\prob \bra{\tau = \infty} = 0$ (note that the limit $Z_{\infty}$ always exists due to the supermartingale convergence theorem). Furthermore, if $A$ denotes the dual optional projection of $\indic_{\dbraco{\tau, \infty}}$, it follows that $\mu:=A + Z$ is a nonnegative uniformly integrable martingale on $\basisp$ with $\mu_t=\expec \bra{A_{\infty} | \F_t}$, for all $t\geq0$ (see e.g. \cite[Section 8.2]{Nik06}).
Moreover, by the general properties of the dual optional projection (see e.g. \cite[Theorem 5.27]{MR1219534}), for any stopping time $\sigma$ on $\basis$, it holds that $\Delta A_{\sigma}=\prob \bra{\tau = \sigma \such \F_{\sigma}}$ on $\set{\sigma < \infty}$.

For all $\nin$, let $\zeta_n \dfn \inf \set{t \in \Real_+ \such Z_t < 1/n}$. Furthermore, set
\begin{equation}\label{zeta}
\zeta \dfn \limn \zeta_n = \inf \set{t \in \Real_+ \such Z_{t-} = 0 \text{ or } Z_t = 0} = \inf \set{t \in \Real_+ \such  Z_t = 0},
\end{equation}
where the last equality holds from the fact that $Z$ is a nonnegative supermartingale on $\basisp$.
We now introduce a stopping time that will be of major importance in the sequel. Consider the $\F_\zeta$-measurable event $\Lambda \dfn \set{\zeta < \infty, \, Z_{\zeta -} > 0, \, \Delta A_{\zeta} = 0}$, and define
\begin{equation}\label{eta}
\eta \dfn \zeta_\Lambda = \zeta \indic_\Lambda + \infty \indic_{\Omega \setminus \Lambda}.
\end{equation}
Clearly, $\eta$ is a stopping time on $\basis$, and it satisfies $\prob\bra{\eta>\tau}=1$. Indeed, $\prob\bra{\tau>\eta | \F_\eta}=Z_\eta=0$ and $\prob\bra{\tau=\eta<\infty | \F_\eta}=\Delta A_\eta \indic_{\{\eta<\infty\}}=\Delta A_\zeta \indic_\Lambda=0$ (remember that $\prob\bra{\tau=\infty}=0$ by assumption).
In \S~\ref{subsec: examples}, it is shown that $\eta$ may be totally inaccessible or accessible. However, Lemma~\ref{lem: key-loc-mart} shows that $\prob \bra{\eta = \sigma < \infty \such \F_{\sigma -} } < 1$ holds for all predictable times $\sigma$ on $\basis$.

The results below establish  stability of condition \naone\ in the current setting of progressive filtration enlargement. Together with their counterparts for initially enlarged filtrations (Theorems~\ref{NA1-init-thm} and~\ref{thm:gen_init}), they are the main results of this paper. 

The first result is concerned with stability of the \naone \ condition for \emph{a fixed} semimartingale model.

\begin{thm}[]\label{thm: NA1_G} 
If \emph{\naonef} holds and $\prob\bra{\eta<\infty,\Delta S_{\eta}\neq0}=0$, then \emph{\naoneg} holds.
\end{thm}

\begin{rem}	\label{rem: NA1_G}
The message of the above theorem is that, to ensure the preservation of \naone \ under progressive filtration enlargement, one only needs to check whether the price process jumps at time $\eta$. It is then clear that, if \naone$(\mathbf{F},\Stilde)$ holds for $\Stilde:=S^{\eta-}=S^{\eta}-\Delta S_{\eta}\indic_{\dbraco{\eta,\infty}}$, then \naoneg \ holds as well, since $\prob\bra{\eta>\tau}=1$. Actually, in order to have \naoneg, it is sufficient that \naone$(\mathbf{F},\Stilde^{\zeta_n})$ holds for all $n\in\N$. Indeed, note that \naone$(\mathbf{F},\Stilde^{\zeta_n})$ implies \naone$(\mathbf{G},S^{\tau\wedge\zeta_n})$, and that the intervals $\dbra{0,\tau\wedge\zeta_n}$ exhaust $\dbra{0,\tau}$, since $\prob\bra{\zeta\geq\tau}=1$. Now the claim follows since the \naone \ condition can be given locally\footnote{Here we provide a proof by way of contradiction. Assume there are $T \in (0 ,\infty)$ and $\chi_T \in \lzp (\G_T)$ such that $\prob \bra{\chi_T > 0} > 0$, satisfying the condition that, for all $x \in (0, \infty)$, there exists $X^x=x+(H^x\cdot S^\tau) \in \X(\bG, S^\tau)$ with $\prob \bra{X^x_T \geq \chi_T} = 1$. Consider the set $A:=\{\chi_T>0\}$ and take $n$ big enough such that $B:=\{\zeta_n\wedge T\geq\tau\wedge T\}\cap A\in\G_T$ satisfies $\prob\bra{B}>0$. Note that $\psi_T:=\chi_T\indic_B\in \lzp (\G_T)$ is such that $\prob \bra{\psi_T > 0} > 0$. Now, for every $x \in (0, \infty)$, define the process $Y^x:=x+(H^x\cdot S^{\tau\wedge\zeta_n}) \in \X(\bG, S^{\tau\wedge\zeta_n})$. By definition of admissibility, $\prob \bra{Y^x_T\geq0}=1$. Moreover, on $B$ we have $Y^x_T=x+(H^x\cdot S^{\tau\wedge\zeta_n})_T=x+(H^x\cdot S^\tau)_T\geq\chi_T=\psi_T$. Altogether this gives $\prob \bra{Y^x_T\geq\psi_T}=1$, which is in contradiction to \naone$(\mathbf{G},S^{\tau\wedge\zeta_n})$.}.
\end{rem}

\begin{rem}	\label{rem:comp_bis}
Define $\widetilde{Z}$ to be the optional projection of $\indic_{\dbra{0, \tau}}$ on $\basisp$ (see also \cite[Section IV.1]{MR604176}); in other words, for any stopping time $\sigma$ on $\basis$, $\widetilde{Z}_\sigma = \prob \bra{\tau \geq \sigma \such \F_{\sigma}}$ holds on $\set{\sigma < \infty}$, so that $\widetilde{Z} = Z + \Delta A$.
It is then straightforward to see that condition $\prob \bra{\eta < \infty, \, \Delta S_{\eta} \neq 0} = 0$ is equivalent to evanescence of the set $\{ Z_{-} > 0, \, \widetilde{Z} = 0, \, \Delta S\neq 0\}$. 
Hence, Theorem~\ref{thm: NA1_G} corresponds exactly to the result proved in \cite[Corollary~2.20, part (b)]{Ch_et_al_13}, by means of different techniques. Moreover, when $S$ is a quasi-left-continuous semimartingale (see \cite[Definition~I.2.25]{MR1943877}), \cite[Theorem~2.8]{Ch_et_al_13} shows that the validity of \naone$(\mathbf{F},\Stilde^{\zeta_n})$, for all $n\in\N$, is actually necessary and sufficient for the preservation of the \naone \ property in $\bG$ (see also \cite[Remark 2.9]{Ch_et_al_13}).
\end{rem}

Theorem~\ref{thm: NA1_G} recovers the already-known fact that condition \naone\ is stable under progressive enlargement for all continuous semimartingales; see \cite{FJS} and \cite{Kre13}.
Moreover, it implies that the condition $\PP\bra{\eta < \infty} = 0$ is sufficient to guarantee \naone\ stability for \emph{any} collection of asset-price processes. In the next result we show that this condition is also necessary in order to have this general stability. In fact, for $\PP\bra{\eta < \infty} > 0$, we provide an explicit example of arbitrage of the first kind, which further shows how condition $\prob \bra{\eta < \infty, \, \Delta S_{\eta} \neq 0} = 0$ in Theorem~\ref{thm: NA1_G} cannot be dropped; see also \S~\ref{subsubsec: inacc}.
Statement (1) of the following theorem is an immediate consequence of Theorem~\ref{thm: NA1_G}, while the proof of statement (2) is given in Section~\ref{subsec:proof_main_prog}.

\begin{thm}\label{thm:gen_progr}
The following statements hold true:
\begin{enumerate}
\item If $\PP\bra{\eta < \infty} = 0$, then for \emph{any} $S$ such that \emph{\naonef} holds, \emph{\naoneg} also holds.
\item Suppose that $\PP\bra{\eta < \infty} > 0$. Then, with $D$ being the predictable compensator of $\indic_{\dbraco{\eta, \infty}}$ on $\basisp$, the nonnegative process $S \dfn \Exp(-D)^{-1} \indic_{\dbraco{0, \eta}}$ is a local martingale on $\basisp$, and $S^\tau$ is nondecreasing with $\prob \bra{S_\tau > 1} > 0$. In particular, condition \emph{\naonef} holds but condition \emph{\naoneg} fails.
\end{enumerate}
\end{thm}

\begin{rem} \label{rem:comp}
Similarly to the discussion in Remark~\ref{rem:comp_bis}, condition $\prob \bra{\eta < \infty} = 0$ is equivalent to evanescence of the set $\{ Z_{-} > 0, \, \widetilde{Z} = 0\} = \{ Z_{-} > 0, \, Z = 0, \, \Delta A = 0\}$. 
Therefore, the characterisation we obtain in Theorem~\ref{thm:gen_progr} is equivalent to that proved in \cite[Theorem~2.22]{Ch_et_al_13}, by means of different techniques.
\end{rem}

Section \ref{sect.progr} is devoted to the proof of Theorem \ref{thm: NA1_G} and Theorem \ref{thm:gen_progr}; several interesting side results are also included there. In \S~\ref{subsec: examples} that follows, a couple of illustrative examples are given. 

\subsection{Main results under initial filtration enlargement} \label{subsec:main_res_init}

We now study the stability of condition \naone\ under an initial enlargement of the filtration $\FF$ with respect to an $\F$-measurable random variable $J$ taking values in a Lusin space $(E,\cE)$, where $\cE$ denotes the Borel $\sigma$-field of $E$. With some abuse of notation, we denote by $\GG=(\G_t)_{t\in \Real_+}$ the right-continuous augmentation of the filtration $\GG^0=(\G^0_t)_{t\in \Real_+}$ defined by $\G^0_t:=\F_t\vee\sigma(J)$, for all $t\in \Real_+$. Let $\gamma : \cE \mapsto [0,1]$ be the law of $J$ (so that $\gamma \bra{B} = \prob \bra{J \in B}$ holds for all $B \in \cE$). Furthermore, for all $t\in \Real_+$, let $\pj_t : \Omega \times \cE \mapsto [0,1]$ be a regular version of the $\F_t$-conditional law of $J$, which exists since $(E,\cE)$ is Lusin.

\begin{ass}	\label{ass:Jac_sep}
Throughout \S \ref{subsec:main_res_init}, we work under the following condition:
\begin{itemize}
	\item[(J)]\quad for all $t\in \Real_+$, $\pj_t \ll\gamma$ holds in the $\PP$-a.s. sense.
\end{itemize}
\end{ass}

Assumption~\ref{ass:Jac_sep} is the classical density hypothesis introduced in \cite{Jac85}. Indeed, as shown in \cite[Proposition 1.5]{Jac85} (see also \cite[Theorem VI.11]{MR1037262}), condition (J) holds if and only if, for all $t\in \Real_+$ there exists a $\sigma$-finite measure $\nu_t$ on $(E,\cE)$ such that $\pj_t \ll \nu_t$ holds in the $\PP$-a.s. sense. 
Jacod's density hypothesis plays a prominent role in financial mathematics, notably in relation to the modeling of additional information (see e.g. \cite{AIS98,GP98,MR1831271,Bau03,GVV06,KH07,KO11}).

The next auxiliary result (the proof of which is postponed to Section~\ref{sect.init}) implies the existence a good version of conditional densities. It essentially corresponds to \cite[Lemma 1.8]{Jac85} (see also \cite[Appendix A.1]{Amen}). Note that $\cO(\FF)$ denotes the $\FF$-optional $\sigma$-field on $\Omega\times \Real_+$.

\begin{lem}	\label{lem-Jac}
There exists a $\pare{\cE \otimes \cO(\FF)}$-measurable function $E \times \Omega \times \Real_+ \ni (x, \omega,t)\mapsto p^x_t(\omega) \in [0, \infty)$, \cadlag \ in $t \in \Real_+$ and such that:
\begin{itemize}
\item[(i)]
for every $t\in \Real_+$, $\pj_t(\ud x) = p^x_t \,\gamma(\ud x)$ holds $\prob$-a.s;
\item[(ii)]
for every $x\in E$, the process $p^x=(p^x_t)_{t\in \Real_+}$ is a martingale on $\basisp$.
\end{itemize}
\end{lem}

For every $x\in E$ and $n\in\N$, define families of stopping times on $\basis$ via
\be	\label{st-init}
\zeta^x_n:=\inf\{t\in\R_+ \such p^x_t<1/n\}
\qquad\text{and}\qquad
\zeta^x:=\inf\{t\in\R_+ \such p^x_t=0\}.
\ee

For all $x\in E$, it holds that $(\zeta^x_n)_{\nin}$ is a nondecreasing sequence, $\prob \bra{\limn \zeta^x_n = \zeta^x} = 1$, and $p^x=0$ on $\lsi\!\zeta^x,\infty\!\lsi$ (see also \cite[Lemma 1.8]{Jac85}).
Note also that, due to \cite[Corollary 1.11]{Jac85}, it holds that $\prob \bra{\zeta^J < \infty} = 0$, with $\zeta^J (\omega) \dfn \zeta^{J(\omega)}(\omega)$ for all $\omega \in \Omega$.
For every $x\in E$, we consider the $\F_{\zeta^x}$-measurable event $\Lambda^x:=\{\zeta^x<\infty,p^x_{\zeta^x-}>0\}$. Define
\be	\label{eta-init}
\eta^x := \zeta^x_{\Lambda^x} = \zeta^x\ind_{\Lambda^x}+\infty\ind_{\Omega\setminus\Lambda^x}, \quad \forall x \in E,
\ee
which is a stopping time on $\basis$ and represents the time at which $p^x$ jumps to zero.

Under Assumption \ref{ass:Jac_sep}, we now discuss counterparts to Theorems~\ref{thm: NA1_G} and~\ref{thm:gen_progr} on the validity of \naone\ in initially enlarged filtrations. Note that Assumption \ref{ass:Jac_sep} guarantees that $S$ is a semimartingale on $\basisgp$, by \cite[Theorem 1.1]{Jac85}, which is proved by relying on the Bichteler-Dellacherie characterisation of semimartingales. (In this respect, see also Remark~\ref{rem:H_init} of the present paper.) This allows us to define the class $\X(\bG, S)$ and the condition NA$_1 (\bG, S)$  as done in \S~\ref{subsec:market_model} with respect to the filtration $\bF$. The first result is concerned with stability of condition \naone\ for \emph{a fixed} semimartingale model.

\begin{thm} \label{NA1-init-thm}
Under Assumption \ref{ass:Jac_sep}, suppose further that the space $\Lb^1 (\Omega, \F, \prob)$ is separable and $\PP\bra{\eta^x<\infty,\Delta S_{\eta^x}\neq0}=0$ holds for $\gamma$-a.e. $x\in E$.  If \emph{\naonef} holds, then \emph{\naone$(\bG,S)$} holds.
\end{thm}

Note that separability is a mild technical assumption that only allows us to use the results of \cite[Proposition 4]{SY}; as the authors of the latter paper mention, it is satisfied in all cases of practical interest.

In \S~\ref{subsubsec: ex_in} we will provide an example showing how condition $\PP\bra{\eta^x<\infty,\Delta S_{\eta^x}\neq0}=0$, for $\gamma$-a.e. $x\in E$, cannot be dropped.

As was the case for progressively enlarged filtrations, Theorem \ref{NA1-init-thm} has the following consequence: if $\PP\bra{\eta^x<\infty}=0$ for $\gamma$-a.e. $x\in E$, condition \naonef \ implies condition \naone$(\bG,S)$ \ for \emph{any} asset-price process $S$. In order to formulate the counterpart to statement (2) of Theorem \ref{thm:gen_progr} (regarding stability of the \naone\  condition for \emph{all} semimartingale models) in the case of initially enlarged filtrations, we have to slightly depart from our original setting. More precisely, the explicit example of an arbitrage of the first kind in the enlarged filtration when $\PP\bra{\eta^x<\infty} > 0$ will involve a potentially infinite collection of semimartingales. (However, see Remark \ref{rem:finite_arb1}.) To wit, with $D^x$ denoting the predictable compensator of $\indic_{\dbraco{\eta^x, \infty}}$ on $\basisp$ for all $x \in E$, define the collection $(S^x)_{x \in E}$ via
\begin{equation} \label{eq:S^x}
S^x \dfn \Exp(-D^x)^{-1} \indic_{\dbraco{0, \eta^x}}, \quad \forall x \in E.
\end{equation}
In Section \ref{sect.init}, under separability assumption on the space $\Lb^1 (\Omega, \F, \prob)$, it is established that one can obtain a version of the function $E \times \Omega \times \Real_+ \ni (x,\omega,t) \mapsto S^x_t(\omega)$ which is  $\cE\otimes\cO(\FF)$-measurable. The process $S^J$ defined via $S^J (\omega, t) \dfn S^{J(\omega)}_t (\omega)$ for all $(\omega, t) \in \Omega \times \Real_+$ is a semimartingale on $\basisgp$, and has the following financial interpretation: an insider with knowledge of $J$ and unit initial capital takes at time zero a position on a single unit of the stock with index $J$, and keeps it indefinitely. Although this strategy may involve an infinite number of assets, it is of the simplest possible buy-and-hold nature. 
Statement (1) of the following theorem is an immediate consequence of Theorem~\ref{NA1-init-thm}, while the proof of statement (2) is given in Section~\ref{subsec:proofs_initial}.

\begin{thm} \label{thm:gen_init}
Under Assumption \ref{ass:Jac_sep}, the following statements hold true:
\begin{enumerate}
\item If $\PP\bra{\eta^x < \infty} = 0$ holds for $\gamma$-a.e $x \in E$, then for \emph{any} $S$ such that \emph{\naonef} holds, \emph{\naone$(\bG,S)$} also holds.
\item Suppose that the space $\Lb^1 (\Omega, \F, \prob)$ is separable and that $\int_E \PP\bra{\eta^x < \infty} \gamma \bra{\ud x} > 0$. Then, the family $(S^x)_{x \in E}$ in \eqref{eq:S^x} consists of local martingales on $\basisp$, and $S^J$ is nondecreasing with $\prob \bra{S^J_t = S^J_0, \, \forall \tir} < 1$.
In particular, \emph{\naone}$(\bF,S^x)$ holds, for every $x\in E$, but \emph{\naone}$(\bG,S^J)$ fails.
\end{enumerate}
\end{thm}

Loosely speaking, in part (2) of Theorem \ref{thm:gen_init}, the insider identifies from the beginning a single asset in the family $\pare{S^x}_{x \in E}$ which will \emph{not} default and can therefore arbitrage.

\begin{rem} \label{rem:finite_arb1}
If $\sum_{k \in \Natural} \prob \bra{J = x_k} = 1$ holds for a family $\set{x_k \such k \in \Natural} \subseteq E$, one can find a single asset that will lead to arbitrage of the first kind. Indeed, $\int_E \PP\bra{\eta^x < \infty} \gamma \bra{\ud x} > 0$ implies that there exists $\kappa \in \Natural$ such that $\prob \bra{\eta^{x_\kappa} < \infty} > 0$. Since $\prob \bra{\zeta^J < \infty} = 0$, $\prob \bra{J = x_\kappa, \, \eta^{x_\kappa} < \infty} = 0$ follows in a straightforward way; therefore, the buy-and-hold strategy $\indic_{\set{J = x_\kappa}}$ results in the arbitrage $\indic_{\set{J = x_\kappa}} \cdot S^{x_\kappa}$. 

When the law $\gamma$ has a diffuse component the previous argument may not work; however, one can still obtain an arbitrage of the first kind using a single asset under an assumption that is stronger (more precisely, at least not weaker) than $\int_E \PP\bra{\eta^x < \infty} \gamma \bra{\ud x} > 0$ as in part (2) of Theorem \ref{thm:gen_init}. To wit, for $B \in \cE$ with $\gamma \bra{B} > 0$, define $\eta^B$ in the obvious way, as the time that the martingale $\pare{\pj_t \bra{B}}_{\tir}$ jumps to zero. Note the equality $\pj_t \bra{B} = \int_B p^x_t \gamma \bra{\ud x}$, for all $\tir$; in particular, $\prob \bra{\eta^B < \infty} > 0$ implies that $\int_E \PP\bra{\eta^x < \infty} \gamma \bra{\ud x} > 0$. (It is an open question whether the converse implication is also true for some set $B\in \cE$.) Under the assumption $\prob \bra{\eta^B < \infty} > 0$ for some $B \in \cE$ with $\gamma \bra{B} > 0$, upon defining $S \dfn \Exp(-D^B)^{-1} \indic_{\dbraco{0, \eta^B}}$ where $D^B$ denotes the predictable compensator of $\indic_{\dbraco{\eta^B, \infty}}$ on $\basisp$, it can be shown that $S$ is a local martingale on $\basisp$, and $\indic_{\set{J \in B}} \cdot S$ is nondecreasing with $\prob \bra{S_t = S_0, \, \forall \tir} < 1$, that is, \naonef\ holds while \naone$(\bG,S)$ fails. 
\end{rem}

\begin{rem}
It is interesting to observe that the necessary and sufficient conditions given in Theorem~\ref{thm:gen_progr} and in Theorem~\ref{thm:gen_init} for the preservation of the \naone \ property under filtration enlargements bear resemblance to the necessary and sufficient condition obtained in \cite{F14} for the preservation of the \naone \ property under absolutely continuous (but not necessarily equivalent) changes of measure. This similarity is not a coincidence, given the deep link existing between filtration enlargements and non-equivalent changes of measure, as shown in \cite{Y85}.
\end{rem}

The proof of Lemma~\ref{lem-Jac} as well as of Theorems~\ref{NA1-init-thm} and \ref{thm:gen_init} is given in \S~\ref{sect.init}. An example in the initial enlargement framework involving the Poisson process is given in \S~\ref{subsec: examples} below.

\subsection{Examples} \label{subsec: examples}

The first two examples are in the progressive filtration enlargement framework. In the first one, the stopping time $\eta$ is totally inaccessible and assertion (2) of Theorem~\ref{thm:gen_progr} is illustrated by explicit computations; the second example contains a set-up where $\eta$ is accessible. 
The last example shows how condition $\PP\bra{\eta^x<\infty,\Delta S_{\eta^x}\neq0}=0$, for $\gamma$-a.e. $x\in E$, cannot be dropped in Theorem~\ref{NA1-init-thm}.

\subsubsection{An example under progressive filtration enlargement where $\eta$ is totally inaccessible} \label{subsubsec: inacc}

Let $\probtriple$ be a complete probability space supporting an $\F$-measurable random variable $\zeta : \Omega \mapsto \Real_+$ such that $\prob \bra{\zeta > t} = \exp(-t)$ holds for all $t \in \Real_+$. Set $\bF = (\F_t)_{t \in \Real_+}$ to be the smallest filtration that satisfies the usual hypotheses and makes $\zeta$ a stopping time. Define $\tau \dfn \zeta / 2$, and consider the filtration $\bG$ obtained as the progressive enlargement of $\bF$ with respect to $\tau$. Let $Z$ and $A$ be defined as in \S~\ref{subsec:main-progr}.

Note that $Z_t = 0$ holds on $\set{\zeta \leq t}$, while $Z_t = \exp(-t)$ holds on $\set{t < \zeta}$, the last fact following from $\tau = \zeta / 2$ and the memoryless property of the exponential law. Therefore, $Z_t = \exp(-t) \indic_{\set{t < \zeta}}$ is true for all $t \in \Real_+$. 
Similarly, $\Delta A_\sigma = \prob \bra{\tau = \sigma \such \F_\sigma} = \prob \bra{\zeta = 2 \sigma \such \F_\sigma} = 0$ is true for all bounded stopping times $\sigma$ on $\basis$, which implies that $\Delta A = 0$. Note that $\zeta = \inf \set{t \in \Real_+ \such Z_{t-} = 0 \text{ or } Z_t = 0}$ and $Z_{\zeta-} = \exp \pare{- \zeta} > 0$. Since $\Delta A = 0$, for $\eta$ defined as in \eqref{eta}, we obtain that $\eta = \zeta$. The predictable compensator of $\indic_{\dbraco{\eta, \infty}}$  on $\basisp$ is equal to $D \dfn (\eta \wedge t)_{t \in \Real_+}$; in particular, $\zeta = \eta$ is totally inaccessible on $\basisp$.

Here we have $\prob \bra{\eta<\infty} = 1$, hence we can proceed to construct a local martingale $S$ as in Theorem~\ref{thm:gen_progr}-(2). To wit, $S \dfn \Exp(-D)^{-1} \indic_{\dbraco{0, \eta}}=\exp(D) \indic_{\dbraco{0, \eta}}$, that is, $S_t = \exp (t) \indic_{\set{t < \zeta}}$ for $t \in \Real_+$. Note that $S$ is a quasi-left-continuous nonnegative martingale on $\basisp$, so that \naonef\ trivially holds. However, since $S$ is \emph{strictly} increasing up to $\tau$, \naoneg \ fails.

\subsubsection{An example under progressive filtration enlargement where $\eta$ is accessible}

Let $\probtriple$ be a complete probability space that supports an $\F$-measurable random variable $\zeta : \Omega \mapsto \Natural$ such that $p_k \dfn \prob \bra{\zeta = k} \in (0, 1)$ holds for all $k \in \Natural$, where $\sum_{k =1}^\infty p_k = 1$. Set $\bF = (\F_t)_{t \in \Real_+}$ to be the smallest filtration that satisfies the usual hypotheses and makes $\zeta$ a stopping time. Since $\zeta$ is $\Natural$-valued, it is an accessible time on $\basisp$. Define $\tau \dfn \zeta - 1$, and consider the progressively enlarged filtration $\bG$. Let $Z$ and $A$ be defined as in \S~\ref{subsec:main-progr}.

Again, one may compute $Z$ explicitly. In fact, $Z_t = 0$ holds on $\set{\zeta \leq t}$; furthermore, upon defining $q_k = \sum_{n = k + 1}^\infty p_n$ for all $k \in \set{0, 1, \ldots}$, and denoting by $\lceil \cdot \rceil$ the integer part, we have 
\[
Z_t = \prob \bra{\tau > t \such \F_t} = \prob \bra{\zeta > t + 1 \such \F_t} = \prob \bra{\zeta > \lceil t + 1 \rceil \such \F_t} = \frac{q_{\lceil t + 1\rceil}}{q_{\lceil t \rceil}}, \quad \text{on } \set{t < \zeta}.
\]
Note that $\zeta = \inf \set{t \in \Real_+ \such Z_{t-} = 0 \text{ or } Z_t = 0}$ and $Z_{\zeta-} = q_{\lceil \zeta \rceil} / q_{\lceil \zeta - 1 \rceil} > 0$.
Furthermore, $\Delta A_\zeta = \prob \bra{\tau = \zeta \such \F_\zeta} = 0$ holds true. It follows that, for $\eta$ defined as in \eqref{eta}, $\eta = \zeta$; in particular, $\eta$ is \emph{accessible} on $\basisp$.

\subsubsection{An example under initial filtration enlargement}  \label{subsubsec: ex_in}

Let us consider a probability space $\probtriple$ supporting a Poisson process $N$ with intensity $\lambda>0$ stopped at time $T\in(0,\infty)$. Let $\FF$ be the right-continuous filtration generated by $N$ and consider the random variable $J\dfn N_T$. As in \cite[\S~4.2]{GVV06} (compare also with \cite[\S~4.3]{MR1831271}), it can be checked that
\[
p^x_t = \e^{-\lambda t}\frac{\bigl(\lambda(T-t)\bigr)^{x-N_t}}{(\lambda T)^x}\frac{x!}{(x-N_t)!}\indic_{\{N_t\leq x\}},\quad \textrm{for all $t\in[0,T)$},
\]
and $p^x_T=\e^{-\lambda T}x!/(\lambda T)^x\indic_{\{N_T=x\}}$, so that Jacod's criterion (Assumption \ref{ass:Jac_sep}) is satisfied. 

Consider then the process $S$ defined by $S_t\dfn\exp\bigl(N_t-\lambda t(\e-1)\bigr)$, for all $t\in[0,T]$. The process $S$ is a strictly positive $\FF$-martingale (see e.g. \cite[Proposition 8.2.2.1]{JYC09}), so that \naone$(\FF,S)$ holds.
However, \naone$(\GG,S)$ does not hold. To see this, define the $\GG$-stopping time $\sigma\dfn\inf\set{t\in[0,T]\such N_t=N_T}$ and consider the strategy $-\indic_{\dbraoc{\sigma,T}}$. Then, for all $t\in[0,T]$, we get
\[
(-\indic_{\dbraoc{\sigma,T}}\cdot S)_t
= \indic_{\{t>\sigma\}}\exp\bigl(N_{\sigma}-\lambda\sigma(\e-1)\bigr)\Bigl(1-\exp\bigl(-\lambda(t-\sigma)(\e-1)\bigr)\Bigr).
\]
In particular, the process $-\indic_{\dbraoc{\sigma,T}}\cdot S$ is nondecreasing and $\PP\bra{\sigma<T}=1$, thus implying that \naone$(\GG,S)$ fails to hold.
Indeed, in the context of the present example, the processes $p^x$ have a positive probability to jump to zero and this event occurs exactly in correspondence of the jump times of the Poisson process $N$, thus showing that the condition $\PP\bra{\eta^x<\infty,\Delta S_{\eta^x}\neq0}=0$ for $\gamma$-a.e. $x\in E$ fail to hold.

\section{Arbitrage of the First Kind in Progressively Enlarged Filtrations}\label{sect.progr}

In this section, the proof of Theorem \ref{thm: NA1_G} and Theorem \ref{thm:gen_progr} will be given. In the process, we will also obtain certain interesting results concerning the behaviour (up to the random time $\tau$) of nonnegative super/local martingales on $\basisp$ in the enlarged filtration $\bG$ (see Section~\ref{subsec: enlarged_loc_marts}). In particular, these results do not follow from classical results of enlargement of filtrations theory.

\subsection{Representation pair associated with $\tau$}
The next result is \cite[Theorem 1.1]{Kar_14}.

\begin{thm} \label{thm: doleans sharpened}
For any random time $\tau$ on $\basis$ satisfying $\prob \bra{\tau=\infty}=0$ there exists a pair of processes $(K, L)$ with the following properties:
\begin{enumerate}
  \item $K$ is $\bF$-adapted, right-continuous, nondecreasing, with $0 \leq K \leq 1$.
  \item $L$ is a nonnegative local martingale on $\basisp$, with $L_0 = 1$.
  \item For any nonnegative optional processes $V$ on $\basis$, we have 
\begin{equation} \label{eq: main}
\expec[V_\tau] = \expec \bra{\int_{\zi} V_t \,  L_t \ud K_t}.
\end{equation}
  \item 
$\int_{\zi}\indic_{\set{K_{t-} = 1}} \ud L_t = 0$ and $\int_{\zi} \indic_{\set{L_t  = 0}} \ud K_t = 0$ hold $\prob$-a.s.
\end{enumerate}
\end{thm}

It also comes as part of the results in \cite[\S 1.1]{Kar_14} that $Z = L (1 - K)$, which gives a particular multiplicative \emph{optional} decomposition of $Z$. In general, there are many possible optional multiplicative decompositions; the properties described in Theorem~\ref{thm: doleans sharpened} specify the pair $(K, L)$ in a unique way.
Note also that, in the special case where $\prob\bra{\tau=\sigma}=0$ for every stopping time $\sigma$ on $\basisp$, the decomposition $Z=L(1-K)$ coincides with the multiplicative Doob-Meyer decomposition of the supermartingale $Z$ (see \cite[Remark 1.6]{Kar_14}).

\begin{rem}	\label{rem: cond_L_K}
Let $\sigma$ be a stopping time on $\basis$. For any $B \in \F_\sigma$, \eqref{eq: main} applied to the process $V = \indic_B \indic_{\dbraoo{\sigma, \infty}}$, combined with $Z = L(1- K)$ and the definition of $Z$, implies that
\[
\expec \bra{L_\sigma (1 - K_\sigma) \indic_B} = \expec \bra{Z_\sigma \indic_B} = \expec \bra{V_\tau} = \expec \bra{\indic_B \int_{(\sigma, \infty)} L_t \ud K_t}.
\]
Since the above equality holds for all $B \in \F_\sigma$, it follows that
\begin{equation} \label{eq: cond_L_K}
L_\sigma (1 - K_\sigma) = \expec \bra{\int_{(\sigma, \infty)} L_t \ud K_t \ \Big| \ \F_\sigma}.
\end{equation}
\end{rem}

\begin{rem} \label{rem: L_tau_pos}
Another use of \eqref{eq: main} gives
\[
\prob \bra{L_\tau = 0} = \expec \bra{\int_{\zi} \indic_{\set{L_t = 0}}  L_t \ud K_t} = 0.
\]
Since $L$ is a nonnegative local martingale on $\basisp$, it follows that $\dbra{0, \tau} \subseteq \set{L > 0}$.
\end{rem}

\begin{lem} \label{lem: event_eq}
For $\zeta$ defined in \eqref{zeta}, and $A$ denoting the dual optional projection of $\indic_{\dbraco{\tau, \infty}}$, the following set equality holds:
\begin{equation} \label{eq: event_eq}
\set{\zeta < \infty, \, Z_{\zeta -} > 0, \, \Delta A_{\zeta} = 0} = \set{\zeta < \infty, \ K_{\zeta - } < 1, \, L_{\zeta -} > 0, \, \Delta K_{\zeta} = 0}.
\end{equation}
Furthermore, $L_{\zeta} = 0$ holds on the above event.
\end{lem}

\begin{proof}
Since $Z = L (1 - K)$, $\set{\zeta < \infty, \, Z_{\zeta -} > 0} = \set{\zeta < \infty, \ K_{\zeta - } < 1, \, L_{\zeta -} > 0}$ is immediate. According to the definition of $K$ in \cite[equation (1.1)]{Kar_14}, it follows that, on $\set{\zeta < \infty}$, $\Delta A_{\zeta} = 0$ implies $\Delta K_{\zeta} = 0$. Furthermore, on $\set{\zeta < \infty, \, Z_{\zeta -} > 0}$, $\Delta K_{\zeta} = 0$ implies that $K_{\zeta} = K_{\zeta -} < 1$, which gives that $\Delta A_{\zeta} = 0$ upon using \cite[equation (1.1)]{Kar_14} again. The set-equality \eqref{eq: event_eq} has been established. Finally, note that the fact that $0 = Z_{\zeta} = L_{\zeta} (1 - K_{\zeta})$ implies that $L_{\zeta} = 0$ has to hold on $\set{\zeta < \infty, \ K_{\zeta - } < 1, \, L_{\zeta -} > 0, \, \Delta K_{\zeta} = 0}$.
\end{proof}

\subsection{Results regarding the stopping time $\eta$}

Recall that $\eta = \zeta \indic_\Lambda + \infty \indic_{\Omega \setminus \Lambda}$, where $\Lambda \dfn \set{\zeta < \infty, \, Z_{\zeta -} > 0, \, \Delta A_{\zeta} = 0}$. In view of \eqref{eq: event_eq}, $\Lambda = \set{\zeta < \infty, \ K_{\zeta - } < 1, \, L_{\zeta -} > 0, \, \Delta K_{\zeta} = 0}$. In the proof of the next result, it is established \emph{inter alia} that $\eta$ is not predictable, when finite.

\begin{lem} \label{lem: key-loc-mart}
Let $D$ be the predictable compensator of $\indic_{\dbraco{\eta, \infty}}$  on $\basisp$. Then:
\begin{enumerate}
	\item $\Delta D < 1$, $\PP$-a.s.; in particular, $\Exp (- D)$ is nonincreasing and strictly positive;
	\item the nonnegative process $\Exp (- D)^{-1} \indic_{\dbraco{0, \eta}}$ is a local martingale on $\basisp$.
\end{enumerate}
\end{lem}

\begin{proof}
For any predictable time $\sigma$ on $\basis$, it holds that $\Delta D_\sigma = \prob \bra{\eta = \sigma \such \F_{\sigma -}}$ on $\set{\sigma < \infty}$ (see e.g. \cite[Theorem 5.27]{MR1219534}). In the next paragraph, we shall show that $\Delta D_\sigma < 1$ holds on $\set{\sigma < \infty}$ for any predictable time $\sigma$ on $\basis$. Then, the predictable section theorem implies that $\Delta D < 1$ $\PP$-a.s.; in particular, the process $\Exp (- D)^{-1} \indic_{\dbraco{0, \eta}}$ will be well-defined. This will establish part (1).

We proceed in showing that $\prob \bra{\eta = \sigma < \infty \such \F_{\sigma -} } < 1$ holds for any fixed predictable time $\sigma$ on $\basis$. Suppose that $\Sigma \dfn \set{\prob \bra{\eta = \sigma < \infty \such \F_{\sigma -} } = 1} \in \F_{\sigma -}$ is such that $\prob \bra{\Sigma} > 0$. Upon replacing $\sigma$ by the predictable time $\sigma_\Sigma \dfn \sigma \indic_{\Sigma} + \infty \indic_{\Omega \setminus \Sigma}$, we infer the existence of a predictable time $\sigma$ on $\basis$ such that $\prob \bra{\sigma < \infty} > 0$ and $\set{\sigma < \infty} = \set{\prob \bra{\eta = \sigma < \infty \such \F_{\sigma -} } = 1}$ hold. From the previous set-equality it follows that $\prob \bra{\eta = \sigma < \infty \such \F_{\sigma -} } = \indic_{\{\eta = \sigma < \infty\}}$, which in particular implies that $\set{\eta = \sigma < \infty} \in \F_{\sigma -}$. Therefore, since $\expec \bra{\Delta (A + Z)_\sigma \such \F_{\sigma-}} = 0$ holds on $\set{\sigma < \infty}$ (because $A + Z$ is a martingale on $\basisp$ and $\sigma$ is predictable on $\basis$),
\[
\expec \bra{\Delta A_\sigma  \such \F_{\sigma -}} = - \expec \bra{\Delta Z_\sigma \such \F_{\sigma -}} = - \expec \bra{\Delta Z_\eta  \such \F_{\sigma -}} = \expec \bra{ Z_{\eta-}  \such \F_{\sigma -}}, \quad \text{on } \set{\eta = \sigma < \infty},
\]
where in the last equality we have used the definition of $\eta$. On the other hand, using again the definition of $\eta$, we obtain that $\expec \bra{\Delta A_\sigma  \such \F_{\sigma -}} = \expec \bra{\Delta A_\eta  \such \F_{\sigma -}} = 0$ holds on $\set{\eta = \sigma < \infty}$. It follows that $\expec \bra{ Z_{\eta-}  \such \F_{\sigma -}} = 0$ on $\set{\eta = \sigma < \infty}$. Since $Z_{\eta -} > 0$ holds on $\set{\eta < \infty}$, the equality $\expec \bra{ Z_{\eta-} \indic_{\set{\eta = \sigma < \infty}} \such \F_{\sigma -}} = 0$ implies that $\prob \bra{\eta = \sigma < \infty} = 0$, which contradicts the fact that $\prob \bra{\sigma < \infty} > 0$ and $\set{\sigma < \infty} = \set{\prob \bra{\eta = \sigma < \infty \such \F_{\sigma -} } = 1}$ hold. Therefore, $\prob \bra{\eta = \sigma < \infty \such \F_{\sigma -} } < 1$ holds for any predictable time $\sigma$ on $\basis$.

We continue in establishing part (2). Let $I = \indic_{\dbraco{\eta, \infty}}$, so that $I - D$ is a local martingale on $\basisp$. Integration-by-parts gives
\[
\Exp (- D)^{-1} \indic_{\dbraco{0, \eta}}  = 1- \int_0^\cdot \Exp (- D)_t^{-1} \ud I_t + \int_0^\cdot \pare{1 - I_{t-} } \ud \Exp (- D)_t^{-1}   = - \int_0^\cdot \Exp (- D)_t^{-1} \ud I_t + \Exp (- D)^{-1},
\]
where the second equality follows from the facts that $1 - I_- = \indic_{\dbra{0, \eta}}$ and $\Exp (- D)^{-1}$ is constant on $\dbraco{\eta, \infty}$. Using It\^o's formula (actually, integration theory for finite-variation processes is sufficient), it is straightforward to check that 
\[
\Exp (- D)^{-1} = 1 + \int_0^\cdot \Exp(- D)^{-1}_t \ud D_t.
\]
It then follows that
\[
\Exp (- D)^{-1} \indic_{\dbraco{0, \eta}}  = 1 - \int_0^\cdot \Exp (- D)_t^{-1} \ud \pare{I - D}_t,
\]
which concludes the argument in view of the fact that $I - D$ is a local martingale on $\basisp$.
\end{proof}

We write $\qprob \sim \prob$ whenever $\qprob$ is a probability that is equivalent to $\prob$ on $\F$. Note that all the quantities that we have defined and depend on $\tau$ (in particular, $\eta$) depend on the underlying probability measure. For establishing Theorem~\ref{thm: NA1_G}, it is important that $\eta$ remains invariant under equivalent changes of probability. The next result ensures that this is indeed the case.

\begin{lem} \label{lem: eta_P_Q}
Let $\qprob \sim \prob$, and let $\eta^\qprob$ be the stopping time on $\basis$ defined under $\qprob$ in analogy to $\eta \equiv \eta^\prob$ defined in \eqref{eta} under $\prob$. Then $\eta^\qprob = \eta$ holds almost surely (under both $\prob$ and $\qprob$).
\end{lem}

\begin{proof}
Denote by $Z^\qprob$ the Az\'ema supermartingale associated with $\tau$ on $\basisq$. We claim that $\{Z^\qprob>0\}=\{Z>0\}$ holds modulo evanescence. Indeed, this follows from the optional section theorem, upon noting that 
\[
\set{Z_\sigma = 0} = \set{\prob \bra{\tau > \sigma \such \F_{\sigma}} = 0} = \set{\qprob \bra{\tau > \sigma \such \F_{\sigma}} = 0} =  \big\{ Z^\qprob_{\sigma} = 0 \big \}
\]
holds for all bounded stopping times $\sigma$ on $\basis$, where the second set-equality holds because $\qprob \sim \prob$.
In particular, $Z_\eta=0$ and $Z_{\eta-}>0$ imply $Z^\qprob_\eta=0$ and $Z^\qprob_{\eta-}>0$.
Now denote by $A^\qprob$ the dual optional projection of $\indic_{\dbraco{\tau, \infty}}$ on $\basisq$. Since $\qprob \sim \prob$ and $\prob \bra{\tau = \eta}=0$, it follows that $\Delta A^\qprob_\eta=\qprob \bra{\tau = \eta \such \F_{\eta}}=0$. Together with the previous observation, this implies $\eta^\qprob\leq\eta$. Upon interchanging the roles of $\prob$ and $\qprob$, one obtains the reverse inequality, completing the proof.
\end{proof}

\subsection{Super/local martingales in the progressively enlarged filtration} \label{subsec: enlarged_loc_marts}

The next result, which will be key in the development, is also of independent interest. 

\begin{prop} \label{prop: loc_mart at random time}
The following statements hold true:
\begin{enumerate}
\item
Let $X$ be a nonnegative supermartingale on $\basisp$. Then, the process $X^{\tau} / L^{\tau}$  is a supermartingale on $\basisgp$.
\item
Let $X$ be a nonnegative local martingale on $\basisp$ such that $\dbraco{\eta, \infty} \, \subseteq \set{X = 0}$ holds (modulo evanescence). Then, the process $X^{\tau} / L^{\tau}$ is a local martingale on $\basisgp$.
\end{enumerate}
\end{prop}

\begin{proof}
Note first that, by Remark~\ref{rem: L_tau_pos}, $1 / L^\tau$ is well defined. 
If $X$ is a nonnegative supermartingale on $\basisp$, the Doob-Meyer decomposition gives that $X=N-B$, where $N$ is a (non-negative) local martingale on $\basisp$ and $B$ is an increasing predictable process on $\basis$ with $B_0=0$. 
Let $s<t$ and $G\in\G_s$. By \eqref{Gfiltr} there exists a set $G_s\in\F_s$ such that $G\cap\set{\tau>s}=G_s\cap\set{\tau>s}$. Define then the nonnegative optional process $Y \dfn \indic_{G_s}\indic_{\dbraoo{s,\infty}}(X^{t}/L^{t})\indic_{\set{L^{t}>0}}$ on $\basis$, so that $\indic_{G\cap\set{\tau>s}}X^{\tau}_t/L^{\tau}_t=Y_{\tau}$. 
In view of Theorem~\ref{thm: doleans sharpened}, it follows that
\begin{equation}	\label{eq:loc_mart_new_proof_1}	\begin{aligned}
\expec\left[Y_{\tau}\right]
&= \expec\left[\int_{[0,\infty)} Y_u L_u \ud K_u\right] \\
&= \expec\left[\indic_{G_{s}}\int_{(s,t]}\frac{X_u}{L_u} \indic_{\{L_u>0\}}L_u \ud K_u
+ \frac{X_{t}}{L_{t}}\indic_{G_{s}\cap\set{L_{t}>0}}\int_{(t,\infty)} L_u \ud K_u\right]	\\
&= \expec\left[\indic_{G_{s}}\int_{(s,t]}X_u\indic_{\{L_u>0\}} \ud K_u
+ \frac{X_{t}}{L_{t}}\indic_{G_{s}\cap\set{L_{t}>0}}L_{t}(1-K_{t})\right] \\
&= \expec\left[\indic_{G_{s}}\left(\int_{(s,t]}X_u \ud K_u
+ X_{t}\indic_{\set{L_{t}>0}}(1-K_{t})\right)\right],
\end{aligned}	\end{equation}
where \eqref{eq: cond_L_K} was used in the third equality above. 
Noting that $\set{L_{t}>0}\subseteq\set{L_{s}>0}$, integration-by-parts then implies that
\begin{equation}	\label{eq:loc_mart_new_proof_2}	\begin{aligned}
\expec\left[Y_{\tau}\right]
& \leq \expec\left[\indic_{G_{s}\cap\set{L_{s}>0}}\left(\int_{(s,t]}X_u \ud K_u
+ X_{t}(1-K_{t})\right)\right] \\
&= \expec\left[\indic_{G_{s}\cap\set{L_{s}>0}}\left(X_{s}(1-K_{s})+\int_{(s,t]}(1-K_{u-}) \ud X_u\right)\right].	
\end{aligned}	\end{equation}
Furthermore, since $0\leq K\leq 1$ and the process $B$ is increasing, it holds that 
\begin{equation}	\label{eq:loc_mart_new_proof_4}
\int_{(s,t]}(1-K_{u-})\ud X_u=\int_{(s,t]}(1-K_{u-})\ud N_u-\int_{(s,t]}(1-K_{u-})\ud B_u	
\leq \int_{(s,t]}(1-K_{u-})\ud N_u
\end{equation}
and $\bigl[(1-K_-) \cdot N, \, (1-K_-) \cdot N\bigr] \leq [N, N]$.
Suppose first that $N\in\mathcal{H}^1$, i.e., $\expec\bigl[[N,N]_{\infty}^{1/2}\bigr]<\infty$, from which it follows that
\begin{equation}	\label{eq:loc_mart_new_proof_3}
\expec \bra{\bigl[(1-K_-) \cdot N, \, (1-K_-) \cdot N\bigr]^{1/2}_{\infty} } \leq \expec \bra{ [N, N]_{\infty}^{1/2} } < \infty.
\end{equation}
Together with \eqref{eq:loc_mart_new_proof_4}, this implies that $\expec\left[\indic_{G_s\cap\set{L_s>0}}\int_{(s,t]}(1-K_{u-})\ud X_u\right]\leq 0$. 
Hence, due to \eqref{eq:loc_mart_new_proof_2},
\begin{align*}
 \expec\left[\indic_{G_{s}\cap\set{\tau>s}}\frac{X^{\tau}_t}{L^{\tau}_t}\right]
&= \expec\left[Y_{\tau}\right]
 \leq \expec\left[\indic_{G_{s}\cap\set{L_{s}>0}}X_{s}(1-K_{s})\right] 
= \expec\left[\indic_{G_{s}\cap\set{L_{s}>0}}\frac{X_{s}}{L_{s}}L_{s}(1-K_{s})\right]	\\
&= \expec\left[\indic_{G_{s}\cap\set{L_{s}>0}}\frac{X_{s}}{L_{s}}\indic_{\set{\tau>s}}\right]
= \expec\left[\indic_{G_{s}\cap\set{\tau>s}}\frac{X_{s}}{L_{s}}\right],
\notag
\end{align*}
where $L_{s}(1-K_{s})=Z_{s}=\prob[\tau>s \such \F_{s}]$ and $\dbra{0,\tau}\subseteq\set{L>0}$ were used in the last line.
Since
\begin{align*}
\expec\left[\indic_G\frac{X^{\tau}_t}{L^{\tau}_t}\right]
&= \expec\left[\indic_{G_{s}\cap\set{\tau>s}}\frac{X^{\tau}_t}{L^{\tau}_t}\right]
+ \expec\left[\indic_{G\cap\set{\tau\leq s}}\frac{X^{\tau}_s}{L^{\tau}_s}\right] \\
&\leq \expec\left[\indic_{G_{s}\cap\set{\tau>s}}\frac{X^{\tau}_s}{L^{\tau}_s}\right]
+ \expec\left[\indic_{G\cap\set{\tau\leq s}}\frac{X^{\tau}_s}{L^{\tau}_s}\right]
= \expec\left[\indic_G\frac{X^{\tau}_s}{L^{\tau}_s}\right],
\end{align*}
we have thus proved that $X^{\tau}/L^{\tau}$ is a supermartingale on $\basisgp$. The general case follows by localization. In fact, by~\cite[Theorem IV.51]{MR1037262}, every local martingale $N$ on $\basisp$ admits a nondecreasing sequence $(\sigma_n)_{\nin}$ of stopping times (under $\bF$ and, \emph{a fortiori}, under $\bG$) $\prob$-a.s. converging to infinity such that $N^{\sigma_n}\in\mathcal{H}^1$ for all $\nin$. The preceding arguments imply that $X^{\sigma_n\wedge\tau}/L^{\tau}$ is a supermartingale on $\basisgp$ for all $\nin$ and statement (1) of the proposition then follows by Fatou's lemma.

In order to prove part (2), define the nondecreasing sequence $(\sigma_n)_{\nin}$ of stopping times via $\sigma_n \dfn \inf \set{t \in \Real_+ \such  [X, X]_t > n} \wedge \zeta_n$, for all $\nin$. For future reference, note that $\sigma_n \leq \zeta \leq \eta$ holds for all $\nin$.
It is straightforward to check that $\limn \prob \bra{\zeta_n < \tau} = 0$; therefore, in order to prove the result, it suffices to show that $X^{\tau \wedge \sigma_n} / L^{\tau \wedge \sigma_n}$ is a martingale on $\basisgp$ for all $\nin$. 
By part (1), the process $X^{\tau \wedge \sigma_n} / L^{\tau \wedge \sigma_n}$ is a supermartingale on $\basisgp$ for all $\nin$.
It follows that it suffices to show that $\expec \bra{X_{\tau  \wedge \sigma_n} / L_{\tau  \wedge \sigma_n}} = \expec\bra{X_0}$ holds for all $\nin$. 
Similarly as in the first part of the proof, set $Y \dfn (X / L) \indic_{\set{L > 0}}$, and note that $Y^{\sigma_n}$ is optional on $\basis$ and $X_{\tau \wedge \sigma_n} / L_{\tau \wedge \sigma_n} = Y^{\sigma_n}_{\tau}$ holds for all $\nin$. 
Computations analogous to \eqref{eq:loc_mart_new_proof_1} allow then to show that
\begin{equation}	\label{comp-mart-progr}
\expec \bra{ \frac{X_{\tau \wedge \sigma_n}}{L_{\tau \wedge \sigma_n}}} 
=\expec\bra{Y^{\sigma_n}_{\tau}}
= \expec \bra{\int_{[0,\sigma_n]}  X_t  \ud K_t +  X_{\sigma_n} \indic_{\set{L_{\sigma_n} > 0}} (1 - K_{\sigma_n}) }.
\end{equation}
Note that $X_{\sigma_n} \indic_{\set{L_{\sigma_n} = 0}} (1 - K_{\sigma_n}) = 0$ holds for all $\nin$; indeed, this follows from Lemma~\ref{lem: event_eq} since $\set{L_{\sigma_n} = 0, \, K_{\sigma_n} < 1} = \set{\sigma_n = \eta}$  holds for all $\nin$. Therefore, similarly as in \eqref{eq:loc_mart_new_proof_2}, integration-by-parts yields that
\[
\expec \bra{ \frac{X_{\tau \wedge \sigma_n}}{L_{\tau \wedge \sigma_n}}} = \expec \bra{\int_{[0,\sigma_n]}  X_t  \ud K_t +  X_{\sigma_n}  (1 - K_{\sigma_n}) } 
= \expec \bra{\int_{[0,\sigma_n]}  (1-K_{t-})  \ud X_t  }
=\expec\bra{X_0},
\]
where the last equality makes use of inequality \eqref{eq:loc_mart_new_proof_3} (now applied with respect to the martingale $X^{\sigma_n}$, with the convention $K_{0-}=0$), for all $\nin$. This completes the argument.
\end{proof}

Proposition \ref{prop: loc_mart at random time} shows that, up to a normalisation with respect to $1/L^{\tau}$, the supermartingale property can always be transferred from the original filtration $\bF$ to the enlarged filtration $\bG$ and provides a sufficient criterion for transforming $\bF$-local martingales into $\bG$-local martingales. As shown in Section~\ref{subsec:proof_main_prog}, this result will play a key role in proving Theorem \ref{thm: NA1_G}. 

In the rest of this section we provide a couple of interesting side results which, though not used in the sequel, are intimately connected to Proposition \ref{prop: loc_mart at random time}. The first one provides a characterisation of the  local martingale property of $X^{\tau}/L^{\tau}$ on $\basisgp$ for \emph{every} nonnegative local martingale $X$ on $\basisp$.

\begin{prop}	\label{mart-progr}
The following statements are equivalent:
\begin{enumerate}
\item For every nonnegative local martingale $X$ on $\basisp$, the process $X^{\tau}/L^{\tau}$ is a local martingale on $\basisgp$.
\item The process $1/L^{\tau}$ is a local martingale on $\basisgp$.
\item $\prob[\eta < \infty]=0$.
\end{enumerate}
\end{prop}
\begin{proof}
Implication (1)~$\Rightarrow$~(2) is trivial, while (3)~$\Rightarrow$~(1) follows from part (2) of Proposition \ref{prop: loc_mart at random time}.
In order to prove (2)~$\Rightarrow$~(3), note that the sequence $\{\tau_n\}_{\nin}$ defined by $\tau_n \dfn \inf\set{t\in\Real_+ \such 1/L^{\tau}_t>n}$, for all $\nin$, is a localising sequence for $1/L^{\tau}$ on $\basisgp$. Define the sequence $\{\nu_n\}_{n\in\N}$ of stopping times on $\basis$ via $\nu_n \dfn \inf \set{t \in \Real_+ \such L_t < 1/n}$, for all $\nin$, and observe that $\tau_n=\nu_n\indic_{\set{\nu_n\leq\tau}}+\infty\indic_{\set{\nu_n>\tau}}$. Then, by computations analogous to \eqref{comp-mart-progr}, we obtain
\[
1 = \expec\left[\frac{1}{L_{\tau\wedge\tau_n}}\right]
= \expec\left[\frac{1}{L_{\tau\wedge\nu_n}}\right]
= \expec\left[K_{\nu_n}+\indic_{\set{L_{\nu_n}>0}}(1-K_{\nu_n})\right]
= 1 - \expec\left[\indic_{\set{L_{\nu_n}=0}}(1-K_{\infty})\right],
\]
where in the last equality we have used the fact that $K$ does not increase on $\set{L=0}$. In turn, this implies that $\set{K_{\infty}<1}\cap\set{L_{\nu_n}=0}=\emptyset$ holds (modulo evanescence). Due to Lemma \ref{lem: event_eq} and since $\set{\Delta K>0}\subseteq\set{L>0}$ holds modulo evanescence (see \cite{Kar_14}), this implies that $\prob \bra{\eta < \infty} = 0$.
\end{proof}

Part (1) of Proposition~\ref{prop: loc_mart at random time} leads to a quick and easy proof of the classical result of \cite{MR519998} on the semimartingale property of $X^{\tau}$ on $\basisgp$ for any semimartingale $X$ on $\basisp$.

\begin{cor}	\label{cor:H_prog}
For any semimartingale $X$ on $\basisp$, the process $X^\tau$ is a semimartingale on $\basisgp$.
\end{cor}
\begin{proof}
Let $X$ be a semimartingale on $\basisp$, so that $X=X_0+B+N$, for some adapted process of finite variation $B$ and a local martingale $N$ on $\basisp$. By \cite[Proposition I.4.17]{MR1943877}, it holds that $N=N'+N''$, where $N'$ and $N''$ are two local martingales on $\basisp$ such that $|\Delta N'|\leq a$ $\prob$-a.s. for some $a>0$ and $N''$ is of finite variation. In order to prove the claim it suffices to show that $(N')^{\tau}$ is a semimartingale on $\basisgp$. To this effect, let $\sigma_n:=\inf\{t\geq0:|N'_t|\geq n\}$, for $n\in\N$, so that $N'_{t\wedge\sigma_n}\geq -(a+n)$ $\prob$-a.s. for all $t\geq0$. Hence, by part (1) of Proposition \ref{prop: loc_mart at random time}, the process $(a+n+N')^{\sigma_n\wedge\tau}/L^{\sigma_n\wedge\tau}$ is a supermartingale on $\basisgp$. In turn, this implies the semimartingale property of $(N')^{\sigma_n\wedge\tau}$ on $\basisgp$. Since semimartingales are stable by localization (see e.g. \cite[Proposition I.4.25]{MR1943877}), this shows the semimartingale property of $(N')^{\tau}$ on $\basisgp$.
\end{proof}

\subsection{Condition \naone\ in the progressively enlarged filtration} \label{subsec:proof_main_prog}

As a consequence of Proposition~\ref{prop: loc_mart at random time}, a \emph{sufficient} condition for \naoneg \ to hold is immediate. The proof of the following result is straightforward, hence omitted. The notation $\Ygp$ is self-explanatory.

\begin{prop} \label{prop: suff_na1}
Suppose that there exists a local martingale deflator $M$ for $S$ on $\basisp$ such that $\set{M > 0} = \dbraco{0 ,\eta}$. Then, 
$M^\tau / L^\tau \in \Ygp$.
\end{prop}

In particular, observe that Proposition \ref{prop: suff_na1} provides an explicit procedure for transforming a local martingale deflator for $S$ on $\basisp$ into a local martingale deflator for $S^{\tau}$ on $\basisgp$.
We are now ready to present the proofs of our results on \naone\ stability under progressive filtration enlargement.

\begin{proof}[Proof of Theorem~\ref{thm: NA1_G}]
In view of Lemma~\ref{lem: eta_P_Q} and Theorem~\ref{thm: num_loc_mart}, we may assume without loss of generality (replacing $\prob$ with $\qprob$ if necessary) the existence of a strictly positive $\Xhat \in \X(\bF, S)$ such that $Y \dfn (1 / \Xhat) \in \Yfp$. Since $\prob \bra{\eta < \infty, \Delta S_{\eta} \neq 0} = 0$ holds, we obtain $\prob \bra{\eta < \infty, \Delta Y_{\eta} \neq 0} = 0$; in particular, $\prob \bra{\eta < \infty, \Delta (Y S)_{\eta} \neq 0} = 0$ holds. In the notation of Lemma~\ref{lem: key-loc-mart}, define $M \dfn Y \Exp(-D)^{-1} \indic_{\dbraco{0, \eta}}$. Note that $M_0 = 1$ and $\set{M > 0} = \dbraco{0 ,\eta}$. By Lemma~\ref{lem: key-loc-mart}, it follows that $M S^i - \bra{ \Exp(-D)^{-1} \indic_{\dbraco{0, \eta}}, Y S^i }$ is a local martingale on $\basisp$ for all $\iii$. Furthermore,
\[
\bra{ \Exp(-D)^{-1} \indic_{\dbraco{0, \eta}}, Y S^i } = \bra{ \Exp(-D)^{-1}, Y S^i } - \bra{ \Exp(-D)^{-1} \indic_{\dbraco{\eta, \infty}}, Y S^i } = \bra{ \Exp(-D)^{-1}, Y S^i },
\]
where $\bra{ \Exp(-D)^{-1} \indic_{\dbraco{\eta, \infty}}, Y S^i } = 0$ follows from the fact that $\Exp(-D)^{-1} \indic_{\dbraco{\eta, \infty}} = \Exp(-D)^{-1}_{\eta} \indic_{\dbraco{\eta, \infty}}$
is a single-jump process, jumping at $\eta$. Since $\Exp(-D)^{-1}$ is predictable, it follows that
\[
\bra{ \Exp(-D)^{-1} \indic_{\dbraco{0, \eta}}, Y S^i } = \bra{ \Exp(-D)^{-1}, Y S^i } = \int_0^\cdot \Delta \Exp(-D)^{-1}_t \ud \pare{Y S^i}_t
\]
is a local martingale on $\basisp$ for all $\iii$. Therefore, $M S^i$ is a local martingale on $\basisp$ for all $\iii$, and Theorem~\ref{thm: NA1_G} follows from Proposition~\ref{prop: suff_na1}.
 \end{proof}

\begin{proof}[Proof of Theorem \ref{thm:gen_progr}]
Statement (1) follows directly from Theorem~\ref{thm: NA1_G}.

For statement (2), let $D$ be as in Lemma~\ref{lem: key-loc-mart}, and define $S = \Exp(-D)^{-1} \indic_{\dbraco{0, \eta}}$.
Then $S_0=1$ and $S$ is a nonincreasing process up to $\tau$, thus $S_\tau\geq1$. Moreover, by Lemma~\ref{lem: key-loc-mart}, $S$ is a local martingale on $\basisp$, hence \naone$(\bF,S)$ holds.
From \eqref{eq: main} and $Z=L(1-K)$, and using integration by parts and the definition of $D$, we have
\begin{align*}
\expec[D_\tau]&=\expec\left[\int_0^\infty D_t L_t \ud K_t\right]=-\expec\left[\int_0^\infty D_t \ud Z_t\right]=\expec\left[\int_0^\infty Z_{t-} \ud D_t\right]=\expec\left[\int_0^\infty Z_{t-} \ud \indic_{\{\eta\leq t\}}\right]\\
&=\expec[Z_{\eta-}\indic_{\{\eta<\infty\}}].
\end{align*}
Therefore, if $\prob\bra{\eta<\infty}>0$, then $\prob\bra{D_\tau>0}>0$, hence $\prob\bra{S_\tau>1}>0$.
This means that \naone$(\bG,S^\tau)$ fails, concluding the proof.
\end{proof}

Note that, in view of Proposition \ref{mart-progr}, Theorem \ref{thm:gen_progr} implies that \naone\ is stable for all semimartingale models if and only if the process $1/L^{\tau}$ is a local martingale on $\basisgp$.

\begin{rem}	\label{rem:NA1_progr_takaoka}
Proposition~\ref{mart-progr} allows to give a direct proof of statement (1) of Theorem~\ref{thm:gen_progr}. Indeed, in view of~\cite[Theorem 2.6]{Tak_13}, \naone$(\bF,S)$ is equivalent to the existence of a process $Y\in\Yfp$. Due to Proposition~\ref{mart-progr}, if $\prob\bra{\eta<\infty}=0$, then $Y^{\tau}/L^{\tau}$ and $(Y^{\tau}/L^{\tau})S^{\tau}$ are local martingales on $\basisgp$, so that $Y^{\tau}/L^{\tau}\in\Ygp$. Hence, by \cite[Theorem 2.6]{Tak_13}, \naone$(\bG,S^{\tau})$ holds. 

Note, however, that this line of reasoning cannot be applied to prove Theorem~\ref{thm: NA1_G}. In fact, in order to construct a strictly positive local martingale deflator in the enlarged filtration $\bG$ starting from an element of $\Yfp$ and relying on Proposition~\ref{prop: loc_mart at random time}, one needs to show that \naone$(\bF,S)$ and $\prob\bra{\eta<\infty,\Delta S_{\eta}\neq0}=0$ together imply the existence of a strictly positive local martingale deflator which does not jump at $\eta$. For this property to hold, we need a more precise statement of the main result of~\cite{Tak_13} in the form of Theorem~\ref{thm: num_loc_mart}.
\end{rem}

\subsection{A partial converse to Proposition~\ref{prop: suff_na1}}

While Proposition~\ref{prop: suff_na1} is sufficient for establishing the \naone\ stability under progressive enlargement in Theorem~\ref{thm: NA1_G}, here we address the inverse problem. Precisely, we seek conditions ensuring the existence of a deflator for $S$ in $\bF$ once a deflator for $S^\tau$ exists in the enlarged filtration $\bG$. Additionally, we want the deflator in $\bF$ to vanish on $\dbraco{\eta, \infty}$, in order to end up in the setting of Proposition~\ref{prop: suff_na1}. The next result shows that this is indeed the case when $\tau$ avoids all stopping times on $\basisp$, meaning that $\prob \bra{\tau = \sigma < \infty} = 0$ holds for all stopping times $\sigma$ on $\basis$.

\begin{thm} \label{thm: NA1_F}
Suppose that $\tau$ avoids all stopping times on $\basisp$. If $\Ygp \neq \emptyset$, then there exists a local martingale deflator $Y$ for $S$ on $\basisp$, 
with $Y=0$ on $\dbraco{\eta, \infty}$.
\end{thm}

\begin{proof}
Let $C$ be the predictable compensator of $\indic_{\dbraco{\tau, \infty}}$ on $\basisg$, and note that for every predictable time $\sigma$ in $\basisg$ it holds that $\Delta C_\sigma = \prob \bra{\tau = \sigma \such \G_{\sigma -}}$ on $\set{\sigma < \infty}$. Now, by assumption $\tau$ avoids all stopping times on $\basisp$, hence in particular all the predictable ones, which is equivalent to say that $\tau$ is a totally inaccessible stopping time on $\basisgp$; see \cite[p.65]{MR604176}. From this fact it follows that $\Delta C_\sigma =0$ holds on $\set{\sigma < \infty}$ for every predictable time $\sigma$ in $\basisg$. The predictable section theorem then implies that $C$ is continuous, thus, in particular, the process $\Exp (- C)^{-1} \indic_{\dbraco{0, \tau}}$ is well-defined. Now, by the same arguments used in the proof of Lemma~\ref{lem: key-loc-mart}, it holds that $\Exp (- C)^{-1} \indic_{\dbraco{0, \tau}}$ is a local martingale on $\basisgp$. Take $M \in \Ygp$. Since $\tau$ avoids all stopping times on $\basisp$, then $\Delta S_\tau=0$ and, as in the proof of Theorem~\ref{thm: NA1_G}, we can assume without loss of generality that $\Delta (M S)_\tau=0$ as well. These two facts allow us to repeat the same steps as in the proof of Theorem~\ref{thm: NA1_G} to show that $U:=M\Exp (- C)^{-1} \indic_{\dbraco{0, \tau}}$ is a local martingale deflator for $S^\tau$ on $\basisgp$.

Now, define $Y$ as the optional projection of $U$ on $\basisp$. Note that $Y_0 = 1$ and that $Y=0$ on $\dbraco{\eta, \infty}$, since $\prob \bra{\tau<\eta} = 1$ (see the discussion after \eqref{eta}). Let $(\sigma'_n)_{n\in\N}$ be a localising sequence for $U$ on $\basisg$, and let $(\sigma_n)_{n\in\N}$ be a sequence of stopping times on $\basis$ such that $\sigma'_n \wedge \tau = \sigma_n \wedge \tau$ for $n\in\N$. Then it is easily verified that $Y$ is a local martingale on $\basisp$, with  $(\sigma_n)_{n\in\N}$ as a localising sequence. Moreover, for any stopping time $\sigma$ in $\basis$ we have
\[
\expec[S^i_{\sigma\wedge\sigma_n} Y_{\sigma\wedge\sigma_n}]=\expec[S^i_{\sigma\wedge\sigma_n} U_{\sigma\wedge\sigma_n}]=\expec[(S^i)^\tau_{\sigma\wedge\sigma_n} U_{\sigma\wedge\sigma_n}]=\expec[(S^i)^\tau_{\sigma\wedge\sigma'_n} U_{\sigma\wedge\sigma'_n}]=S^i_0, \quad \forall \iii.
\]
This shows that $YS^i$ is a local martingale on $\basisp$ for all $\iii$ and concludes the proof.
\end{proof}

\section{Arbitrage of the First Kind in Initially Enlarged Filtrations}\label{sect.init}
In this section, the proof of Theorem \ref{NA1-init-thm} and Theorem \ref{thm:gen_init} will be given, and interesting side results will also be discussed. The validity of Jacod's criterion (Assumption \ref{ass:Jac_sep}) is tacitly assumed throughout.
We start by proving the existence of a good version of conditional densities for $J$.
\begin{proof}[Proof of Lemma~\ref{lem-Jac}]
Denote by $\cO(\oFF)$ the optional $\sigma$-field associated to the filtration $\oFF =(\oF_t)_{t\in \Real_+}$ on $E \times \Omega$ defined by $\oF_t:=\bigcap_{s>t}\left(\cE\otimes\F_s\right)$, $t\in \Real_+$. Note that $\cE\otimes\cO(\FF) \subseteq \cO(\oFF)$ (see \cite{Jac85}). By  \cite[Lemma 1.8]{Jac85}, Assumption \ref{ass:Jac_sep} implies the existence of an $\cO(\oFF)$-measurable nonnegative function $\tilde{p}:(x,\omega,t)\mapsto\tilde{p}^x_t(\omega)$ such that $(i)$-$(ii)$ hold. Since, for every $x\in E$, the process $\tilde{p}^x$ is $\FF$-optional, being $\FF$-adapted and \cadlag, Remark 1 after Proposition 3 of \cite{SY} gives the existence of a $\cE\otimes\cO(\FF)$-measurable version $p$ of $\tilde{p}$.
\end{proof}

The following consequence of Lemma~\ref{lem-Jac} will be used in several places: for any $t\in \Real_+$ and $\pare{\cE\otimes\F_t}$-measurable function $E \times \Omega \times \Real_+ \ni (x, \omega,t)\mapsto f^x_t(\omega) \in \Real_+$, it holds that
\be	\label{exp-init}
\EE\left[f_t^J\right] 
= \EE\left[\int_Ef_t^x\,p^x_t\,\gamma[\ud x]\right]
= \int_E\EE\left[f_t^x\,p^x_t\right]\gamma[\ud x].
\ee

\subsection{Results regarding the stopping times $\pare{\eta^x}_{x \in E}$}
The next result can be regarded as a counterpart to Lemma~\ref{lem: key-loc-mart} in the case of initially enlarged filtrations. 
Note that $\cP(\FF)$ denotes the $\FF$-predictable $\sigma$-field on $\Omega\times \Real_+$ in all that follows.

\begin{lem}	\label{compensator}
Fix $x\in E$, and let $D^x$ be the predictable compensator of $\ind_{\lsi\eta^x,\infty\lsi}$ on $\basisp$, with $\eta^x$ defined in \eqref{eta-init}. Then:
\begin{enumerate}
	\item $\Delta D^x < 1$, $\PP$-a.s.; in particular, $\Exp (- D^x)$ is nonincreasing and strictly positive;
	\item the nonnegative process $\Exp (- D^x)^{-1} \indic_{\dbraco{0, \eta^x}}$ is a local martingale on $\basisp$.
\end{enumerate}
Suppose moreover that the space $\Lb^1 (\Omega, \F, \prob)$ is separable.  Then, the function $E \times \Omega \times \Real_+ \ni (x,\omega,t)\mapsto\Exp(-D^x)_t(\omega)$ can be chosen $\cE\otimes\cP(\FF)$-measurable.
\end{lem}

\begin{rem}\label{rmk:sep}
Note that separability of $\Lb^1 (\Omega, \F, \prob)$ is only needed to ensure that the collection $\pare{\Exp(-D^x)}_{x \in E}$ introduced in Lemma~\ref{compensator} admits a version with good measurability properties.
\end{rem}

\begin{proof}
Fix $x\in E$. For any $\FF$-predictable time $\sigma$, it holds that $\Delta D^x_{\sigma}=\PP\bra{\eta^x=\sigma|\F_{\sigma-}}$ on $\{\sigma<\infty\}$. As in the proof of Lemma~\ref{lem: key-loc-mart}, if the event $\{\PP\bra{\eta^x=\sigma<\infty|\F_{\sigma-}}=1\}$ has positive probability, one can find an predictable time $\tilde{\sigma}$ on $\basis$ such that $\PP\bra{\eta^x=\tilde{\sigma}<\infty}>0$ and $\{\eta^x=\tilde{\sigma}<\infty\}\in\F_{\tilde{\sigma}-}$. Then, by the martingale property of $p^x$ on $\basisp$ and the definition of $\eta^x$,
\[
0 = \EE\left[\Delta p^x_{\tilde{\sigma}}|\F_{\tilde{\sigma}-}\right]
= \EE\left[\Delta p^x_{\eta^x}|\F_{\tilde{\sigma}-}\right]
= -\EE\left[p^x_{\eta^x-}|\F_{\tilde{\sigma}-}\right],
\qquad\text{on } 
\{\eta^x=\tilde{\sigma}<\infty\}.
\]
In turn, since $p^x_{\eta^x-}>0$ holds on $\{\eta^x<\infty\}$, this implies that $\PP\bra{\eta^x=\tilde{\sigma}<\infty}=0$, thus leading to a contradiction and showing that $\PP\bra{\eta^x=\sigma<\infty|\F_{\sigma-}}<1$ holds in the $\PP$-a.s. sense for any predictable time $\sigma$ on $\basis$. Part (1) then follows by the predictable section theorem, while part (2) can be proved by relying on the same arguments used in the proof of Lemma~\ref{lem: key-loc-mart}.

Finally, since $\Lb^1(\Omega,\F,\PP)$ is assumed separable, \cite[Proposition 4]{SY} gives the existence of a $\cE\otimes\F\otimes\cB(\R_+)$-measurable function $(x,\omega,t)\mapsto D^x_t(\omega)$ such that, for all $x\in E$, $D^x$ is the predictable compensator of $\ind_{\lsi\eta^x,\infty\lsi}$ on $\basisp$. Due to \cite[Remark 1, after Proposition 3]{SY}, the function $D$ can also be chosen $\cE\otimes\cP(\FF)$-measurable and the same measurability property is inherited by the function $(x,\omega,t) \mapsto \Exp(-D^x)_t(\omega)$ (see also \cite[\S~12]{SY}).
\end{proof}

In order to establish our main results, we need to ensure that the collection $(\eta^x)_{x \in E}$ of stopping times on $\basis$ remains invariant under equivalent changes of measure, for $\gamma$-a.e. $x\in E$. 

\begin{lem}	\label{eta-inv}
Let $\QQ$ be a probability measure on $(\Omega,\F)$ with $\QQ\sim\PP$. For $x \in E$, let $\eta^{\QQ,x}$ be the stopping time on $\basis$ defined under $\QQ$ in analogy to $\eta^{\PP,x}:=\eta^x$ defined in \eqref{eta-init} under $\PP$. Then $\eta^{\QQ,x}=\eta^x$ holds almost surely (under both $\PP$ and $\QQ$) for $\gamma$-a.e. $x\in E$.
\end{lem}
\begin{proof}
As can be readily checked, Assumption \ref{ass:Jac_sep} is invariant under equivalent changes of probability. Hence, there exists a nonnegative $\cE\otimes\cO(\FF)$-measurable function $E \times \Omega \times \Real_+ \ni (x,\omega,t)\mapsto q_t^x(\omega)$ satisfying the properties of Lemma~\ref{lem-Jac} under $\QQ$. Moreover, due to \cite[Corollary 1.11]{Jac85} (now applied under the probability $\QQ$), it holds that $\QQ\bra{q^J_t=0}=0$ and also $\PP\bra{q^J_t=0}=0$, since $\QQ\sim\PP$, for all $t\in \Real_+$. Hence, by using formula \eqref{exp-init} applied to the $\cE\otimes\F_t$-measurable function $f^x_t=\ind_{\{q^x_t=0\}}$, for $t\in \Real_+$, we obtain
\[
0 = \PP\bra{q^J_t=0}
= \int_E\EE\left[\ind_{\{q^x_t=0\}}p^x_t\right]\gamma \bra{\ud x},
\qquad\text{for all }t\in \Real_+,
\]
so that $\{q^x_t=0\}\subseteq\{p^x_t=0\}$ $\PP$-a.s. for $\gamma$-a.e. $x\in E$. In a similar way, due to the symmetric role of $\PP$ and $\QQ$, one can show that $\{p^x_t=0\}\subseteq\{q^x_t=0\}$ holds $\QQ$-a.s. for $\gamma$-a.e. $x\in E$ and for all $t\in \Real_+$. By right-continuity, $\{q^x=0\}=\{p^x=0\}$ holds (up to evanescence), for $\gamma$-a.e. $x\in E$. Hence, by definition, $\eta^{\QQ,x}=\eta^x$ holds almost surely (under both $\PP$ and $\QQ$) for $\gamma$-a.e. $x\in E$.
\end{proof}

\subsection{Super/local martingales in the initially enlarged filtration} \label{subsec: enlarged_loc_marts_init}

The next result is a counterpart to Proposition~\ref{prop: loc_mart at random time} in the case of initially enlarged filtrations.  Recall that $\PP\bra{\zeta^J=\infty}=1$, as explained after \eqref{st-init}, so that the optional process $1/p^J$ on $\basisgp$ is well-defined.

\begin{prop}	\label{mart-init}
Let $X: E \times \Omega \times \Real_+ \mapsto \Real_+$ be $\cE\otimes\cO(\FF)$-measurable, and such that $X^x$ is \cadlag \ for every $x\in E$.
Then the following statements hold.
\begin{enumerate}
\item
If $X^x$ is a supermartingale on $\basisp$ for $\gamma$-a.e. $x\in E$, then, $X^J/p^J$ is a supermartingale on $\basisgp$.
\item
If $X^x$ is a local martingale on $\basisp$ and $\lsi\!\eta^x,\infty\!\lsi\subseteq\{X^x=0\}$ (modulo evanescence) hold for $\gamma$-a.e. $x\in E$, then, $X^J/p^J$ is a local martingale on $\basisgp$.
\end{enumerate}
\end{prop}

\begin{proof}
We first prove part (1).
For any $s\leq t$, $A\in\F_s$ and  $h:(E,\cE)\rightarrow(\R_+,\cB_{\R_+})$, using the fact that $\prob \bra{\zeta^J=\infty} = 1$ together with formula \eqref{exp-init} (with $f_t(x)=\ind_{A\cap\{\zeta^x>t\}}\,g(x)X^x_t/p^x_t$) and the supermartingale property of $X^x$ on $\basisp$, for $\gamma$-a.e. $x\in E$, we obtain
\begin{align*}
\expec\left[\ind_A\,g(J)\frac{X_t^J}{p^J_t}\right]
&= \expec\left[\ind_{A\cap\{\zeta^J>t\}}g(J)\frac{X_t^J}{p^J_t}\right] \\
&= \int_Eg(x)\expec\left[\ind_{A\cap\{\zeta^x>t\}}X^x_t\right] \gamma [\ud x]	\\
&\leq \int_Eg(x)\expec\left[\ind_{A\cap\{\zeta^x>s\}}X^x_t\right] \gamma[\ud x] \\
&\leq \int_Eg(x)\expec\left[\ind_{A\cap\{\zeta^x>s\}}X^x_s\right] \gamma[\ud x]
= \expec\left[\ind_A\,g(J)\frac{X_s^J}{p^J_s}\right].
\end{align*}
By the monotone class theorem, this shows that $X^J/p^J$ is a supermartingale on $(\Omega, \, \GG^0, \, \prob)$. By right-continuity, this implies the supermartingale property on $\basisgp$, thus proving part (1).

To prove part (2), note first that, since $X^x$ is a nonnegative local martingale on $\basisp$, hence a supermartingale on $\basisp$, the sequence $\pare{\sigma^x_n}_{n\in\N}$ defined by $\sigma^x_n:=\inf\{t\in\R_+\,|\,X^x_t>n\}$ for $n \in \N$ is localising for $X^x$ on $\basisp$, for $\gamma$-a.e. $x\in E$. Moreover, since $X$ is $\cE\otimes\cO(\FF)$-measurable, the function $E \times \Omega \ni (x,\omega)\mapsto\sigma_n^x(\omega)\wedge t$ is $\cE\otimes\F_t$-measurable for all $t\in \Real_+$ and $n\in\N$, and, as a composition of measurable mappings, the function $E \times \Omega \ni (x,\omega)\mapsto X^x_{\sigma^x_n(\omega)\wedge t}(\omega)$ is also $\cE\otimes\F_t$-measurable, for all $t\in \Real_+$ and $n\in\N$ (compare also with \cite{SY}, Remark 1 after Theorem 2). Since $p$ is $\cE\otimes\cO(\FF)$-measurable (see Lemma~\ref{lem-Jac}), the same reasoning allows to show that the function $E \times \Omega \ni (x,\omega)\mapsto X^x_{\sigma^x_n(\omega)\wedge\zeta^x_n(\omega)\wedge t}(\omega)/p^x_{\sigma^x_n(\omega)\wedge\zeta^x_n(\omega)\wedge t}(\omega)$ is $\cE\otimes\F_t$-measurable for all $t\in \Real_+$ and $n\in\N$, where the stopping time $\zeta^x_n$ on $\basis$ is defined in \eqref{st-init}. Then, for any $t\geq0$,  formula \eqref{exp-init} gives
\[	\ba
\EE\left[\frac{X^J_{\sigma^J_n\wedge\zeta^J_n\wedge t}}{p^J_{\sigma^J_n\wedge\zeta^J_n\wedge t}}\right]
&= \int_E\EE\left[
\frac{X^x_{\sigma^x_n\wedge\zeta^x_n\wedge t}}{p^x_{\sigma^x_n\wedge\zeta^x_n\wedge t}}
\ind_{\left\{p^x_{\sigma^x_n\wedge\zeta^x_n\wedge t}>0\right\}}p^x_t\right]\gamma[\ud x]	\\
&= \int_E\EE\left[
X^x_{\sigma^x_n\wedge\zeta^x_n\wedge t}
\ind_{\left\{p^x_{\sigma^x_n\wedge\zeta^x_n\wedge t}>0\right\}}\right]\gamma[\ud x]	\\
&= \int_E\EE\left[
X^x_{\sigma^x_n\wedge\zeta^x_n\wedge t}\right]\gamma[\ud x]	\\
&= \int_E\EE\left[
X^x_0\right]\gamma[\ud x]	
= \EE\left[
\frac{X^J_0}{p^J_0}\right],
\ea	\]
where the second equality follows from the martingale property of $p^x$ on $\basisp$ for all $x\in E$, the third equality from the  fact that $\{p^x_{\sigma^x_n\wedge\zeta^x_n\wedge t}=0\}=\{\eta^x=\sigma^x_n\wedge\zeta^x_n\wedge t\}$ together with the assumption that $\lsi\!\eta^x,\infty\!\lsi\subseteq\{X^x=0\}$ for $\gamma$-a.e. $x\in E$, the fourth equality from the martingale property of $X^x_{\sigma^x_n\wedge\cdot}$ on $\basisp$, for $\gamma$-a.e. $x\in E$ and $n\in\N$, and the last equality from all the previous steps in reverse order. 
In turn, since by part (1) the process $(X^J/p^J)^{\sigma^J_n\wedge\zeta^J_n}$ is a supermartingale on $\basisgp$ for all $\nin$, this implies that $(X^J/p^J)^{\sigma^J_n\wedge\zeta^J_n}$ is a martingale on $\basisgp$ for all $\nin$.
Since $\prob \bra{ \limn \sigma^x_n = \infty } = 1$ holds for every $x\in E$, and $\PP\bra{\zeta^J=\infty}=1$, 
the sequence $\pare{\sigma^J_n\wedge\zeta^J_n}_{n\in\N}$ of stopping times on $\basisg$ satisfies $\prob \bra{\limn \pare{\sigma^J_n\wedge\zeta^J_n} = \infty} = 1$, thus proving that $X^J/p^J$ is a local martingale on $\basisgp$.
\end{proof}

A result analogous to part (1) of Proposition \ref{mart-init} has been recently established in \cite{ImkPer13} (see their Proposition 5.2). More specifically, according to their terminology, the process $1/p^J$ is a \emph{universal supermartingale density} for $\GG$. 

\begin{rem}\label{rem:H_init}
Proposition \ref{mart-init} can be used to establish that any semimartingale $X$ on $\basisp$ remains a semimartingale on $\basisgp$. As was the case in Corollary \ref{cor:H_prog}, it suffices to show the result whenever $X$ is a nonnegative and bounded local martingale, thus a supermartingale, on $\basisp$. By part (1) of Proposition \ref{mart-init}, the process $X / p^J$ is a a semimartingale on $\basisgp$; since also $1 / p^J$ is a strictly positive semimartingale on $\basisgp$, the result follows.
\end{rem}

We proceed with a result that is a ramification of Proposition \ref{mart-init} (this side result will not be used in other places). In the same spirit of Proposition \ref{mart-progr}, we can characterise the local martingale property of $X^J/p^J$ on $\basisgp$ for \emph{every} $\cE\otimes\cO(\FF)$-measurable non-negative function $X$ such that $X^x$ is a local martingale on $\basisp$ for $\gamma$-a.e. $x\in E$.

\begin{prop}	\label{all-mart-init}
The following statements are equivalent:
\begin{enumerate}
\item For any $X: E \times \Omega \times \Real_+ \mapsto \Real_+$ that is $\cE\otimes\cO(\FF)$-measurable, and such that $X^x$ is \cadlag \ for every $x\in E$, $X^x$ is a local martingale on $\basisp$ and $\lsi\!\eta^x,\infty\!\lsi\subseteq\{X^x=0\}$ (modulo evanescence) hold for $\gamma$-a.e. $x\in E$, the process $X^J/p^J$ is a local martingale on $\basisgp$.
\item The process $1/p^J$ is a local martingale on $\basisgp$.
\item $\prob\bra{\eta^x < \infty} = 0$ holds for $\gamma$-a.e. $x\in E$.
\end{enumerate}
\end{prop}
\begin{proof}
Implication (1)~$\Rightarrow$~(2) is trivial, while (3)~$\Rightarrow$~(1) follows from part (2) of Proposition \ref{mart-init}.
In order to prove (2)~$\Rightarrow$~(3), note that the sequence $\pare{\zeta^J_n}_{\nin}$ of stopping times on $\basisg$ is localising for $1/p^J$ (see \eqref{st-init}), so that $\expec[1/p^J_{\zeta^J_n\wedge T}]=\expec[1/p^J_0]$, for any $T\in\Real_+$. Hence, due to formula \eqref{exp-init} applied first to the $\cE\otimes\F_0$-measurable function $E \times \Omega \ni (x,\omega)\mapsto\indic_{\{p^x_0(\omega)>0\}} \big( 1/p^x_0(\omega) \big)$ and then to the $\cE\otimes\F_t$-measurable function $E \times \Omega \ni (x,\omega)\mapsto\indic_{\{p^x_{\zeta^x_n\wedge T}>0\}} \big( 1/p^x_{\zeta^x_n\wedge T} \big)$,
\begin{align*}
\int_E\expec\left[\indic_{\{p^x_0>0\}}\right]\gamma(\ud x)
= \expec\left[\frac{1}{p^J_0}\right]
= \expec\left[\frac{1}{p^J_{\zeta^J_n\wedge T}}\right]
&= \int_E\expec\left[\frac{1}{p^x_{\zeta^x_n\wedge T}}\indic_{\{p^x_{\zeta^x_n\wedge T}>0\}}p^x_T\right]\gamma [\ud x]	\\
&= \int_E\expec\left[\indic_{\{p^x_{\zeta^x_n\wedge T}>0\}}\right]\gamma[\ud x],
\end{align*}
where in the last equality we have used the martingale property of $p^x$ on $\basisp$ for every $x\in E$.
This implies that $\{p^x_0>0\}\cap\{p^x_{\zeta^x_n\wedge T}=0\} = \emptyset$ holds (modulo evanescence) for $\gamma$-a.e. $x\in E$, for all $T\in\R_+$. Equivalently, it holds that $\prob\bra{\eta^x=\infty}=1$ for $\gamma$-a.e. $x\in E$.
\end{proof}

\subsection{Condition \naone\ in the initially enlarged filtration}
\label{subsec:proofs_initial}

In the spirit of Proposition~\ref{prop: suff_na1}, we can then establish a sufficient condition for the validity of \naone\ in the initially enlarged filtration $\GG$. The proof of the next proposition is a straightforward application of Proposition~\ref{mart-init}. The notation $\Y(\GG, S, \prob)$ is clear.

\begin{prop}	\label{NA1-init-prop}
Suppose that there exists a $\cE\otimes\cO(\FF)$-measurable function $M : E \times \Omega \times \Real_+ \mapsto \Real_+$ such that $M^x_0=1$ and $M^x$ is \cadlag, for every $x\in E$,  $M^x$ and $M^x S$ are local martingales on $\basisp$ and $\{M^x>0\}\subseteq\lsi\!0,\eta^x\!\lsi$ hold for $\gamma$-a.e. $x\in E$. Then, $M^J/p^J\in\Y(\GG, S, \prob)$.
\end{prop}

We are now in the position to prove our first main theorem in the  framework of initial filtration enlargement.

\begin{proof}[Proof of Theorem \ref{NA1-init-thm}]
We follow the proof of Theorem~\ref{thm: NA1_G} in the case of progressively enlarged filtrations.
In view of Theorem~\ref{thm: num_loc_mart} and Lemma~\ref{eta-inv}, we may assume without loss of generality the existence of a strictly positive $\widehat{X}\in\mathcal{X}(\FF,S)$ such that $Y:=1/\widehat{X}\in\mathcal{Y}(\FF,S,\PP)$. Since  $\PP\bra{\eta^x<\infty, \, \Delta S_{\eta^x}\neq0}=0$ holds for $\gamma$-a.e. $x\in E$, we obtain $\PP\bra{\eta^x<\infty, \, \Delta Y_{\eta^x}\neq0}=0$ and $\PP\bra{\eta^x<\infty,\, \Delta(YS)_{\eta^x}\neq0}=0$ for $\gamma$-a.e. $x\in E$. In the notation of Lemma~\ref{compensator}, define the function $E \times \Omega \times \Real_+ \ni (x,\omega,t)\mapsto M^x_t(\omega) \dfn Y_t(\omega)\Exp(-D^x)_t^{-1}(\omega)\ind_{\{\eta^x(\omega)>t\}}$. For all $x\in E$, the process $M^x$ is \cadlag\ and satisfies $M^x_0=1$ and $\{M^x>0\}=\dbraco{0,\eta^x}$. By part (2) of Lemma~\ref{compensator} and proceeding as in the proof of Theorem~\ref{thm: NA1_G}, it can be shown that $M^x$ and $M^xS$ are local martingales on $\basisp$ for $\gamma$-a.e. $x\in E$. Moreover, due to the separability of $\Lb^1(\Omega,\F,\PP)$, Lemma~\ref{compensator} shows that $\Exp(-D)$ admits a $\cE\otimes\cP(\FF)$-measurable version. Since $\cP(\FF)\subseteq\cO(\FF)$, the conclusion then follows from Proposition~\ref{NA1-init-prop}.
\end{proof}

Finally, we provide the proof of our last main result.

\begin{proof}[Proof of Theorem \ref{thm:gen_init}]
Statement (1) follows directly from Theorem~\ref{NA1-init-thm}, by Remark~\ref{rmk:sep} and since $\PP\bra{\eta^x < \infty} = 0$ implies that $D^x$ is indistinguishable from $1$.
We proceed with the proof of statement (2). Due to Lemma \ref{compensator}, the function $E \times \Omega \times \Real_+ \ni (x,\omega,t)\mapsto S^x_t(\omega) \dfn \Exp( D^x)^{-1}_t(\omega)\indic_{\{\eta^x(\omega)>t\}}$ is $\cE\otimes\cP(\FF)$-measurable, and, therefore, also $\cE\otimes\cO(\FF)$-measurable. Moreover, for all $x\in E$, $S^x$ is a local martingale on $\basisp$. Recall that $\prob \bra{\eta^J=\zeta^J=\infty} = 1$ (see \S~\ref{subsec:main_res_init}), so that the process $S^J$ is nondecreasing. Moreover, using in sequence formula \eqref{exp-init}, integration by parts and the properties of predictable compensators, we get, for any $T\in(0,\infty)$,
\[
\expec\left[D^J_T\right]
= \int_E\expec\left[D^x_Tq^x_T\right] \gamma[\ud x]
= \int_E\expec\left[\int_{(0,T]}q^x_{t-}dD^x_t\right]\gamma[\ud x]
= \int_E\expec\left[q^x_{\eta^x-}\indic_{\{\eta^x\leq T\}}\right]\gamma[\ud x].
\]
Hence, if $\int_E \PP\bra{\eta^x < \infty} \gamma \bra{\ud x} > 0$, then $\PP\bra{D^J_T>0}>0$ holds for some $T\in(0,\infty)$, which implies that $\prob \bra{S^J_t = S^J_0, \, \forall \tir} < 1$.
\end{proof}

Note that, in view of Proposition \ref{all-mart-init}, the \naone\ stability (in the sense of Theorem \ref{thm:gen_init}) in the enlarged filtration $\GG$ is also equivalent to the local martingale property of  $1/p^J$ on $\basisgp$.

\begin{rem}	\label{rem:NA1_init_takaoka}
We want to point out that, analogously to Remark~\ref{rem:NA1_progr_takaoka}, Proposition~\ref{all-mart-init} allows to give a direct proof of statement (1) of Theorem~\ref{thm:gen_init}. Indeed, in view of~\cite[Theorem 2.6]{Tak_13}, \naone$(\bF,S)$ is equivalent to the existence of a process $Y\in\Yfp$. Due to Proposition~\ref{all-mart-init}, if $\prob\bra{\eta^x<\infty}=0$ holds for $\gamma$-a.e. $x\in E$, then $Y/p^J$ and $(Y/p^J)S$ are local martingales on $\basisgp$, meaning that $Y/p^J\in\Y(\GG, S, \prob)$. \cite[Theorem 2.6]{Tak_13} then implies that \naone$(\bG,S)$ holds.
However, as for the case of progressive enlargement, this argument fails to provide a direct proof of Theorem~\ref{NA1-init-thm} (compare with Remark~\ref{rem:NA1_progr_takaoka}).
\end{rem}

\bibliographystyle{alpha}
\bibliography{AFK_na1}
\end{document}